\documentclass[11pt,reqno]{amsart}

\usepackage[T1]{fontenc}
\usepackage[utf8]{inputenc}
\usepackage{amsmath}
\usepackage{amssymb}
\usepackage{amsthm}
\usepackage{array}
\usepackage{booktabs}
\usepackage{needspace}
\usepackage[margin=1.1in]{geometry}

\usepackage[hidelinks]{hyperref}
\hypersetup{pdfauthor={Enkai Zhang},pdftitle={Sharp order-preserving integer models for short additive equalities}}

\theoremstyle{plain}
\newtheorem{theorem}{Theorem}[section]
\newtheorem{proposition}[theorem]{Proposition}
\newtheorem{lemma}[theorem]{Lemma}
\newtheorem{corollary}[theorem]{Corollary}
\newtheorem{question}[theorem]{Question}
\theoremstyle{definition}
\newtheorem{definition}[theorem]{Definition}
\newtheorem{remark}[theorem]{Remark}

\theoremstyle{plain}
\newtheorem{mainthm}{Theorem}

\newcommand{\R}{\mathbb{R}}
\newcommand{\Z}{\mathbb{Z}}
\newcommand{\Q}{\mathbb{Q}}
\newcommand{\mass}{\mathrm{m}}
\newcommand{\height}{\operatorname{height}}
\newcommand{\diam}{\operatorname{diam}}

\newcommand{\spn}{\operatorname{span}}
\newcommand{\conv}{\operatorname{conv}}

\makeatletter
\let\Models@linkfile\hyper@linkfile
\def\hyper@linkfile#1#2#3{\begingroup
  \let\Hy@href@page\@empty
  \Models@linkfile{#1}{#2}{#3}\endgroup}
\makeatother

\theoremstyle{plain}
\newtheorem{Stheorem}{Theorem}[section]
\theoremstyle{plain}

\theoremstyle{plain}

\theoremstyle{plain}

\theoremstyle{plain}

\theoremstyle{definition}

\theoremstyle{definition}
\newtheorem{Sremark}[Stheorem]{Remark}
\theoremstyle{definition}

\theoremstyle{plain}

\providecommand{\MainPrefix}{}
\providecommand{\SuppPrefix}{}
\hypersetup{hypertexnames=false}
\begin{document}
\hypertarget{article-start}{}
\renewcommand{\SuppPrefix}{Supplement }

\raggedbottom

\title[Sharp ordered integer models for short additive equalities]
      {Sharp order-preserving integer models \\ for short additive equalities}

\author{Enkai Zhang}
\address{University of Toronto Scarborough, Toronto, Ontario, Canada}
\email{ek.zhang@mail.utoronto.ca}


\begin{abstract}
For a set of $m$ real numbers in increasing order, we seek an increasing
integer model in which two $q$-term sums are equal if and only if their
counterparts are equal in the original set. Repetitions are allowed, and
unequal sums may change order. Let $H_m(q)$ be the least diameter that
suffices for every such set. We prove $H_4(q)=q(q+1)$ and
$H_5(q)=q^2(q+1)$ for $q\ge2$. In the five-element case, a single primitive
relation of mass $\mu$ admits a model of diameter at most $\mu q(q+1)$.
The proof constructs integer rays along a flag of faces of the ordered
relation cone. Lifting a model on a facet gives
$H_m(q)\le q^{m-2}+H_{m-1}(q)$ and determines the first two asymptotic
terms for fixed $m\ge4$. At order two, exhaustive enumeration gives
$H_6(2)=24$ and $H_7(2)=48$. The same model bounds apply to finite
additive-square-free spectra.
\end{abstract}

\maketitle
\enlargethispage{3pt}

\section{Introduction}\label{M-sec:intro}

\subsection{The problem}
Fix integers $m\ge2$ and $q\ge1$. Let $T=(t_0<\cdots<t_{m-1})$ be a real set written in increasing order.
We seek integers $s_0<\cdots<s_{m-1}$ with small diameter that reproduce
exactly its short additive equalities. More precisely, for every
$1\le h\le q$ and every two lists of indices, repetitions allowed, we require
\[
 \sum_{\nu=1}^h t_{i_\nu}=\sum_{\nu=1}^h t_{j_\nu}
 \quad\Longleftrightarrow\quad
 \sum_{\nu=1}^h s_{i_\nu}=\sum_{\nu=1}^h s_{j_\nu}.
\]
We call $S=(s_0,\ldots,s_{m-1})$ an order-preserving integer $q$-model of $T$.
Adding the same $q-h$ terms to both sides of an equality shows that testing
$h=q$ suffices. The elements retain their order, while two unequal sums may
reverse order.

Translating $s_0$ to zero and dividing by the gcd of the remaining entries
preserve the requirements. We therefore consider primitive models beginning
at zero. Let $H_m(q)$ be the least integer $D$ such that every ordered real
$m$-element set has a model of diameter at most $D$. For $q\ge2$, existence follows from Theorem~\ref{M-thm:flag-model};
for $q=1$, any strictly increasing integer set has the required profile. The word ``alphabet'' will also be used for such a set when we
discuss sequences of its elements.

\paragraph{Example.}
Take $q=2$ and $T=(0,\theta,1,2)$ with $0<\theta<1$ irrational.
Its nontrivial two-term equality is $0+2=1+1$. The ordered integer set
$(0,1,3,6)$ preserves this equality as $0+6=3+3$ and introduces no other
nontrivial two-term equality. Theorem~\ref{M-thmE} also shows that diameter six is
necessary for this example. The problem is to control such a diameter
uniformly for every ordered set, including sets with different collections
of short equalities.

\subsection{Main results}
\begin{mainthm}\label{M-thmA}
$H_4(q)=q(q+1)$ for every $q\ge2$, and $H_4(1)=3$.
\end{mainthm}

\begin{mainthm}\label{M-thmB}
$H_5(q)=q^2(q+1)$ for every $q\ge2$.
\end{mainthm}

\begin{mainthm}\label{M-thmC}
$H_6(2)=24$ and $H_7(2)=48$.
\end{mainthm}

For four and five elements the bounds hold uniformly in $q$.
The lower examples insert one irrational number into an integer geometric
progression. For the upper bounds we work in the cone cut out by the
original equalities and the weak ordering of the elements. A model on an
ordering facet can be lifted into its interior by adding one primitive
integer ray. This gives the following recurrence.

\begin{mainthm}[general bounds]\label{M-thmF}
For $m\ge3$ and $q\ge2$,
\[
 q^{m-2}+q^{m-3}\le H_m(q)\le q^{m-2}+H_{m-1}(q),
\]
with $H_2(q)=1$. Consequently $H_m(q)\le\sum_{j=2}^{m-2}q^j$ for $m\ge5$,
\[
 H_m(q)=q^{m-2}+q^{m-3}+O_m(q^{m-4})\qquad(m\ge4\text{ fixed},\ q\to\infty),
\]
the lower bound is attained for $m=3,4,5$, and
$q^4+q^3\le H_6(q)\le q^4+q^3+q^2$.
\end{mainthm}

The delicate case for five elements is a one-dimensional space of short
relations. If its primitive generator has mass $\mu$, Proposition~\ref{M-prop:rank-one-mass}
constructs a model of diameter at most $\mu q(q+1)$. At order two, the
recurrence reduces the six- and seven-element questions to finite
intervals. Counting the equality patterns that occur as the diameter
increases determines both exact values.

\subsection{Related work}\label{M-sec:intro-shin}\label{M-sec:intro-attribution}\label{M-sec:intro-label}
In Shin's terminology, the equalities in our model form the weak addition-table
type. His Corollary~8.8 realizes both weak and ordered types in
$[0,(m-1)q^{m-2}]$, and Corollary~7.3 gives the sharp diameter $q^{m-2}$
among weak types of maximal relation rank \cite{M-Shin26}. These results
already imply $H_m(q)=\Theta_m(q^{m-2})$. Our estimates refine the
coefficient and the second term, and determine exact radii in the cases
above. The transfer-minor bound and the kernel-cone method come from
Shin's Lemma~7.1 and Theorem~7.2, building on Amirkhanyan, Bush and
Croot \cite{M-ABC18}. Along an ordering face we may choose the signs of
unequal sums anew. This gives the graded ray heights in
Theorem~\ref{M-thm:flag-model}.

The distinction can already be seen with positional weights. For an integer
$B>q$, consider $(0,1,B,\ldots,B^{m-2})$. Every increasing model that
preserves the signs of comparisons of $q$-sums, translated to start at
zero, must satisfy $s_{i+1}>q s_i$ for $1\le i\le m-2$.
Thus its diameter is at least $1+q+\cdots+q^{m-2}$. This is the
real-alphabet instance of the lexicographic weight argument in
\cite[Theorem~7.9]{M-Shin26}.

There is a second distinction at $q=2$. Konyagin and Lev asked whether
every $m$-element integer set with $m\ge7$ has a Freiman-isomorphic
copy in $[0,2^{m-2}]$ when the bijection may permute the elements
\cite[Problem~7.15]{M-CL07}. Requiring the increasing bijection changes
the bound: Theorem~\ref{M-thm:corank-model} forces diameter at least
$3\cdot2^{m-3}$, even for integer input, by taking an integer model
of its extremal weak type. For background on Freiman modeling see
\cite{M-Fre73,M-GR09,M-Gry13,M-KL00,M-Nat18,M-OBr25}.

A coarser invariant retains only $|qT|$, the number of different sums.
Nathanson's interval problem concerns this invariant \cite{M-Nat26}.
Remark~\ref{M-rem:label-length} compares the two diameters, using the
four-point formula in \cite{M-Sumset26}. Section~\ref{M-sec:application}
applies the equality-model bounds to words avoiding additive squares.

\section{Definitions and the transfer-minor lemma}\label{M-sec:prelim}

In the example alphabets below, $\theta$ denotes an irrational number
in $(0,1)$. When both $\theta$ and $\eta$ occur, we assume
$0<\theta<\eta<1$ and that $1,\theta,\eta$ are linearly independent
over $\Q$.

\subsection{Balanced relations, profiles and models}
A balanced vector $\delta\in\Z^m$ has $\sum_i\delta_i=0$. Its mass is
\[
 \mass(\delta)=\sum_i\max(\delta_i,0)=\sum_i\max(-\delta_i,0).
\]
A \emph{$q$-relation} is a balanced vector of mass at most $q$. For an ordered
alphabet $T=(t_0<\cdots<t_{m-1})$ we write $\delta\cdot T=\sum_i\delta_it_i$.
Its \emph{$q$-zero profile} is
\[
 Z_q(T)=\{\delta\in\Z^m:\ \textstyle\sum_i\delta_i=0,
               \mass(\delta)\le q,\ \delta\cdot T=0\}.
\]
Thus two ordered alphabets have the same profile precisely when the
increasing correspondence preserves all vanishing $q$-relations in both directions.

\begin{definition}\label{M-def:model}
Let $q\ge1$ be an integer. Two ordered alphabets $T=(t_0<\cdots<t_{m-1})$ and
$T'=(t'_0<\cdots<t'_{m-1})$ are \emph{$q$-equivalent} if
for every balanced vector $\delta\in\Z^m$ of mass at most $q$,
\[
 \delta\cdot T=0\quad\Longleftrightarrow\quad\delta\cdot T'=0
 \qquad(\mass(\delta)\le q)
\]
under the increasing correspondence $t_i\mapsto t'_i$. An
\emph{order-preserving $q$-model} of a real alphabet $T$ is a $q$-equivalent
increasing integer alphabet $T'$.
\end{definition}

This preserves equalities of two blocks of every equal length at most $q$.
Conversely, the positive and negative entries of a balanced vector
give the multiplicities of the letters in two sums, each with
$\mass(\delta)$ terms. Adding the same letters to both sums gives
the equivalent formulation with exactly $q$ terms. Only zero relations are
prescribed: signs of unequal block-sum differences need not be preserved.

\begin{remark}\label{M-rem:freiman}
In the language of Freiman \cite{M-Fre73} and Grynkiewicz \cite{M-Gry13}, $T'$
is an order-preserving $q$-model of $T$ exactly when the increasing bijection
$t_i\mapsto t'_i$ is a Freiman isomorphism of order $q$ between $T$ and $T'$:
by the padding just described, two sums of $q$ letters agree on $T$ if and only
if the corresponding sums agree on $T'$. At $q=2$ this is the order-preserving
Freiman isomorphism of Amirkhanyan, Bush and Croot \cite{M-ABC18}. Shin
\cite{M-Shin26} calls the $q$-zero profile under the increasing correspondence
the weak $q$-addition-table type.
\end{remark}

An integer alphabet is \emph{primitive} if, after translating its first
letter to $0$, the greatest common divisor of its letters is $1$. Its
\emph{diameter} is the difference between its last and first letters. For an
ordered real alphabet $T$ let $M_q(T)$ be the least diameter of an
order-preserving $q$-model of $T$, and for $m\ge1$ let
\[
 H_m(q)=\max\{M_q(T):\ T\text{ an ordered real $m$-alphabet}\}.
\]
The maximum exists because $M_q(T)$ depends only on the $q$-zero profile of
$T$, of which there are finitely many, once models are known to exist
(Theorem~\ref{M-thm:flag-model}). For $k\ge3$ let $L_k(q)$ be the least
diameter of a $k$-element integer alphabet with no nonzero $q$-relation;
such alphabets exist by Lemma~\ref{M-lem:append-qc-plus-one} below and the
four-letter construction in Theorem~\ref{M-thm:sharp-four-model}.

\Needspace{7\baselineskip}
\begin{lemma}\label{M-lem:elementary}
Let $q\ge1$.
\begin{enumerate}
\item Translating an integer model by an integer, multiplying it by a positive
 integer, and reflecting both alphabets ($t\mapsto-t$, which reverses the
 order of the letters) preserve $q$-equivalence.
\item Dividing a model that begins at $0$ by the greatest common divisor of
 its letters preserves $q$-equivalence and strict order and does not
 increase the diameter. In particular a model of least diameter is
 primitive, and every model can be replaced by a primitive one beginning at
 $0$ of no larger diameter.
\item If $T'$ is a $q$-model of $T$, then for every set of labels the
 corresponding sub-alphabet of $T'$ is a $q$-model of the corresponding
 sub-alphabet of $T$.
\item $H_m(q)$ is nondecreasing in $m$; moreover $H_1(q)=0$ and $H_2(q)=1$.
\end{enumerate}
\end{lemma}
\begin{proof}
For a balanced $\delta$ and integers $c>0$, $e$ we have
$\delta\cdot(cT'+e)=c\,\delta\cdot T'$ and $\delta\cdot(-T)=-\delta\cdot T$,
which proves (1) and (2). A balanced relation on a sub-alphabet extends by
zeros to a balanced relation of the same mass on the full alphabet, which
proves (3). For (4), extend any smaller alphabet by distinct larger real
letters, model the extension and restrict to the old letters by (3); after
translating the first value to zero, gcd division can only reduce the
resulting diameter. A one-letter alphabet has the model $\{0\}$. The only
balanced relations on two letters are the multiples of $(1,-1)$, which vanish
on no alphabet of two distinct letters; hence $\{0,1\}$ is a $q$-model of
every two-letter alphabet, and no smaller diameter is possible.
\end{proof}

\subsection{Rank}
Translate $t_0$ to zero and write $x_i=t_i-t_0$ for $1\le i\le m-1$, so that
$0<x_1<\cdots<x_{m-1}$. Delete coordinate zero from a balanced row; since
$\delta_0=-\sum_{i\ge1}\delta_i$, the deleted row $(\delta_1,\ldots,\delta_{m-1})$
determines $\delta$, and $\delta\cdot T=\sum_{i\ge1}\delta_ix_i$. We speak of
the mass of a deleted row, meaning the mass of the balanced vector it
determines. Let $V\subseteq\Q^{m-1}$ be the rational span of the deleted rows
of \emph{all} $q$-relations vanishing on $T$, let $r=\dim V$ be the
\emph{complete short-relation rank} of $T$, and let $d=m-1-r$. The real kernel
$L=\ker_\R V\subseteq\R^{m-1}$ has dimension $d$ and contains the strictly
ordered nonzero vector $x$; hence $r\le m-2$ and $d\ge1$. Choose a basis of
$V$ among the deleted rows of vanishing $q$-relations, with masses
$h_1,\ldots,h_r\le q$, and put $H=\prod_{i=1}^rh_i$, with $H=1$ if $r=0$.

\subsection{The transfer-minor lemma}
The minor estimate below is Shin's transfer bound
\cite[Lemma~7.1]{M-Shin26}; we include the elementary proof to fix notation.
A row of mass $h$ is a sum of $h$ projected unit transfers, each of the form
$\pm e_i$ or $e_i-e_j$. Every square minor of a matrix of such transfers is
$0$ or $\pm1$. Indeed a row with at most one nonzero entry permits induction
by expansion; if every row has two entries, they are $+1$ and $-1$, and the
all-ones column vector lies in the kernel. Multilinearity therefore proves:
\begin{lemma}\label{M-lem:mass-minor}
For any square matrix whose rows are deleted rows of balanced relations of
masses $h_1,\ldots,h_s$, $|\det|\le\prod_{i=1}^sh_i$.
\end{lemma}
\begin{remark}\label{M-rem:mass-minor-rows}
In particular an ordering-boundary row (mass one) contributes a factor one
and a $q$-relation a factor at most $q$. Here an \emph{ordering-boundary
row} is one of the rows $e_1$ or $e_{i+1}-e_i$, which are the deleted rows
of the balanced relations $t_1-t_0$ and $t_{i+1}-t_i$ of mass one.
\end{remark}
Whenever $m-2$ independent integer
rows determine a one-dimensional kernel in $\R^{m-1}$, their signed maximal
minors give an integer generator of that line; dividing by the coordinate gcd
makes it primitive, and its last coordinate is bounded by the corresponding
cofactor in Lemma~\ref{M-lem:mass-minor}. We refer to the last coordinate
$x_{m-1}$ of a vector with nonnegative coordinates as its \emph{height}; for
a model beginning at $0$ the height is the diameter.

\Needspace{8\baselineskip}
\begin{corollary}[one-dimensional kernels]\label{M-cor:max-rank}
Suppose $r=m-2$, so that $d=1$. Then $L$ is a rational line, and $x$ is a
positive multiple of the primitive integer generator $w$ of $L$ with positive
last coordinate. Consequently $T$ is affinely equivalent to the primitive
integer alphabet $\{0,w_1,\ldots,w_{m-1}\}$, which is an order-preserving
model of $T$ preserving all balanced relations, and its diameter satisfies
$w_{m-1}\le H\le q^{m-2}$.
\end{corollary}
\begin{proof}
The $r=m-2$ selected rows have a one-dimensional common kernel, which
contains $x$ and is spanned by their cofactor vector; so $L$ is a rational
line and $x$ is a positive real multiple of $w$. A balanced relation vanishes
on $x$ if and only if it vanishes on $w$. The last coordinate of $w$ is an
$(m-2)\times(m-2)$ cofactor, possibly divided by a common gcd, and
Lemma~\ref{M-lem:mass-minor} bounds it by $H\le q^{m-2}$.
\end{proof}

For four letters, two independent vanishing $q$-relations thus have a
one-dimensional common kernel whose primitive integer generator has last
coordinate at most $q^2$; a strictly ordered real vector in that kernel fixes
the positive orientation. The rank-two bound itself is sharp:
$\{0,q,q+1,q^2\}$ has the independent mass-$q$ relations $qa=c$ and
$qb=a+c$. In general the value $q^{m-2}$ is sharp for the rank $m-2$: the
alphabet $\{0,1,q,\ldots,q^{m-2}\}$ has the $m-2$ independent mass-$q$
relations $q\cdot q^{j}=q^{j+1}$, whose common kernel is the line spanned by
$(1,q,\ldots,q^{m-2})$, so that every order-preserving model has diameter at
least $q^{m-2}$. This is Shin's Corollary~7.3 \cite{M-Shin26}.

\section{Flags in the ordered relation kernel}\label{M-sec:flag}

Let $V$ be the rational span of \emph{all} vanishing $q$-relation rows of an
ordered real alphabet, after deleting coordinate zero, as in
Section~\ref{M-sec:prelim}. Write $r=\dim V$ and $d=m-1-r$. Choose a basis
from those vanishing $q$-relation rows, with masses $h_1,\ldots,h_r\le q$;
put $H=\prod_ih_i$, with $H=1$ for an empty basis.

In $L=\ker_{\R}V$, consider
\[
 K_0=L\cap\{0\le x_1\le\cdots\le x_{m-1}\}.
\]
The original strictly ordered vector makes this a rational pointed cone of
dimension $d$ in $L$. Height $x_{m-1}$ is positive on every nonzero point.
A \emph{face} of $K_0$ is obtained by turning some of the ordering
inequalities $x_1\ge0$, $x_{i+1}-x_i\ge0$ into equalities; the span of a face
of dimension $j$ is defined inside $L$ by $d-j$ independent original ordering
boundaries. We subdivide each face by all $q$-relation hyperplanes, ignoring
restrictions that vanish identically on its span; a \emph{chamber} of a face
$G$ is a closed cone of this finite central subdivision that has the same
dimension as $G$. Every ordering functional and every $q$-relation functional
not identically zero on $\spn G$ has a constant weak sign on a chamber of
$G$, and is nonzero on its relative interior.

\begin{lemma}[chamber extension]\label{M-lem:chamber-extension}
Let $G$ be a face of $K_0$ of dimension at least two, let $F$ be a facet of
$G$, and let $C_F$ be a chamber of $F$. Then there is a chamber $C$ of $G$
with $C\cap F=C_F$; in particular $C_F$ is a facet of $C$.
\end{lemma}
\begin{proof}
Take $p$ in the relative interior of $C_F$ and $y$ in the relative interior
of $G$, outside every nonidentically-zero short hyperplane there. For small
$\varepsilon>0$, $p+\varepsilon y$ keeps every short-relation sign nonzero at
$p$. A short row vanishing at $p$ must vanish on the whole span of $F$;
otherwise it would cut the interior of its chamber. Its restriction to the
span of $G$ is either zero or a nonzero multiple of the facet functional,
hence is nonzero at $y$ in the latter case. Thus the perturbation lies inside
a chamber $C$ of $G$ whose restriction to $F$ is exactly $C_F$: the rows
vanishing on $\spn F$ impose only equalities on $F$, and every other row has
the same nonzero sign at $p+\varepsilon y$ as at $p$, so the defining
inequalities of $C$ restricted to $F$ are exactly those of $C_F$. In
particular $C_F$ is a facet of $C$.
\end{proof}

\begin{theorem}[Flag bound]\label{M-thmD}\label{M-thm:flag-model}
For $m\ge2$ and $q\ge2$, the alphabet has an order-preserving $q$-equivalent
primitive integer model of diameter at most
\[
 H\sum_{j=0}^{d-1}q^j.
\]
Here the basis is chosen from the vanishing integer $q$-relation rows.
The masses in $H$ belong to their full balanced vectors. The space $V$
contains all vanishing $q$-relations.
\end{theorem}
The proof has three steps. First choose compatible chambers along a full flag
of ordering faces. Next select one new integer ray at each dimension and bound
its last coordinate by the transfer-minor estimate. Finally add the selected
rays. They lie in one chamber, so a relation cannot vanish by cancellation;
because they span the relation kernel, no additional short relation vanishes
on their sum. The same reasoning makes all ordering gaps strictly positive.

\begin{proof}
Choose a complete flag of faces
\[
 F_1\subset F_2\subset\cdots\subset F_d=K_0,
 \qquad\dim F_j=j.
\]
Each $F_{j-1}$ is a facet of $F_j$, and the span of $F_j$ is defined
inside $L$ by $d-j$ independent original ordering boundaries. Subdivide each
face by all $q$-relation hyperplanes, ignoring restrictions that vanish
identically on its span.

By Lemma~\ref{M-lem:chamber-extension}, applied inductively, we can choose
full-dimensional closed chambers $C_j\subseteq F_j$ with $C_1=F_1$ and
$C_j\cap F_{j-1}=C_{j-1}$. In particular $C_{j-1}$ is a facet of $C_j$.

Choose the primitive positive ray generator $u_1$ of $F_1$. For each
$j\ge2$, choose a primitive extreme-ray generator $u_j$ of $C_j$ outside
the span of $F_{j-1}$. Such a ray exists because $C_j$ spans $F_j$.
All these rays lie in the final chamber, and
\[
 \spn_{\R}(u_1,\ldots,u_j)=\spn_{\R}F_j.
\]
The $j$-th ray is specified by the $r$ original rows, $d-j$ independent
ordering rows, and $j-1$ further active short-relation or ordering rows.
Indeed, an extreme ray of the $j$-dimensional pointed cone $C_j$, which is
cut out inside $\spn F_j$ by finitely many ordering and $q$-relation
halfspaces, lies on $j-1$ independent hyperplanes of that defining list: if
the rows active at a nonzero point of the ray had a common kernel of
dimension greater than one inside $\spn F_j$, a small perturbation in both
directions within that kernel would preserve every inactive strict
inequality, and the point would be the midpoint of two noncollinear points
of $C_j$, contradicting extremality.
These $m-2$ independent rows leave a line. Cofactors and
Lemma~\ref{M-lem:mass-minor} give
\[
 \height(u_j)\le Hq^{j-1}.
\]
The first ordering rows have mass one, and each additional row has mass
at most $q$. Dividing the cofactor gcd only decreases height.

Put $w=\sum_ju_j$. Each ordering functional and each oriented
nonidentically-zero $q$-relation functional is nonnegative on $C_d$. If
one vanished on $w$, it would vanish on every $u_j$, hence on all of $L$,
contradicting strict ordering or its nonzero restriction. Thus $w$ is
strictly ordered, all original short relations vanish on it, and every
other short relation is nonzero. Dividing its coordinate gcd gives a
primitive model with no larger diameter. Summing the height bounds
proves the theorem.
\end{proof}

Linear independence of the rays suffices here. The chamber containing
their sum may assign different signs to unequal sums in the original set.

\begin{remark}\label{M-rem:flag-attribution}
The skeleton of this proof (kernel cone, primitive extreme rays via
cofactors, heights from Lemma~\ref{M-lem:mass-minor}, a sum of rays in the
relative interior, division by the gcd) is that of Shin's Theorem~7.2
\cite{M-Shin26}, who credits it to Amirkhanyan, Bush and Croot \cite{M-ABC18};
for the sign-preserving problem that theorem gives $dq^{d-1}\prod_ip_i$ with
the birth degrees $p_i$. The ordering-face flag supplies compatible chambers and graded heights
$Hq^{j-1}$, replacing $dq^{d-1}$ by $\sum_{j<d}q^j$. The coefficients agree
for $d=1$, and the latter is smaller for $d\ge2$. This improvement uses
the freedom to change nonzero comparison signs: the chambers $C_j$ are
chosen along the flag rather than around the original alphabet. Since $H\le q^r$,
the theorem gives in particular
\begin{equation}\label{M-eq:rank-sensitive-sum}
 \diam\le\sum_{j=r}^{m-2}q^j\le\sum_{j=0}^{m-2}q^j .
\end{equation}
At $q=2$ the dimension count $d=m-1-r$ is treated by
Konyagin and Lev \cite[Theorem~4]{M-KL00}; compare also Freiman's classical
rank-one estimates \cite{M-Fre73}. For general background on Freiman
homomorphisms see \cite{M-Gry13}.
\end{remark}

\section{Two-dimensional kernels}\label{M-sec:corank}

\begin{theorem}[Two-dimensional kernels]\label{M-thmE}\label{M-thm:corank-model}
If $m\ge3$, $q\ge2$ and the complete short-relation rank is $r=m-3$,
there is an order-preserving primitive $q$-model of diameter at most
\[
 (q+1)H\le q^{m-3}(q+1).
\]
The universal bound $q^{m-3}(q+1)$ under this rank hypothesis is sharp.
\end{theorem}
\begin{proof}
The upper bound is Theorem~\ref{M-thm:flag-model} with $d=2$. For sharpness
fix irrational $0<\theta<1$ and take
\[
 T=\{0,\theta\}\cup\{q^j:0\le j\le m-3\}.
\]
Its power relations have mass $q$ and rank $m-3$, which is the full rational
relation rank by independence of $\theta$ and $1$. Every increasing integer
model must be $\{0,a\}\cup\{q^jb:0\le j\le m-3\}$ with $0<a<b$.
If $b\le q$, the extra balanced relation $(a-b,b,-a,0,\ldots,0)$ has
mass $b$ and vanishes on the model but not on the irrational original.
Hence $b\ge q+1$ and the diameter is at least $q^{m-3}(q+1)$.

Equality is attained by
\[
 \{0,1\}\cup\{q^j(q+1):0\le j\le m-3\}.
\]
Modulo $q+1$, a short relation forces the coefficient of $1$ to be zero,
since its absolute value is at most $q$. Dividing the remaining equation
by $q+1$ gives exactly the original equation on the powers of $q$. Thus
the complete short zero profile is preserved, including $m=3$.
\end{proof}

An alternative proof by lattice subdivision, and the obstruction to
the same lattice step in dimension three, are given in
Supplement Section~\ref{S-supp:lattice}. The flag proof above is
sufficient for all subsequent results.

\begin{remark}\label{M-rem:corank-attribution}
The witness $\{0,\theta,1,q,\ldots,q^{m-3}\}$ is Shin's set
$P_{q,m-1}=\{0\}\cup\{q^j:0\le j\le m-3\}$ \cite[Corollary~7.3]{M-Shin26} with one
irrational letter inserted. The integer equality model
$\{0,1,q+1,q(q+1),\ldots\}$ also occurs in his Theorem~9.1 at $r=k-3$,
$\ell=h+1$, which realizes the prescribed active rank and sumset cardinality.
The argument above establishes the sharp ordered equality-model bound
$q^{m-3}(q+1)$, including the lower bound for every model of the irrational
witness. For the sign-preserving problem, his Theorem~7.2 at $d=2$ gives
$2q\prod_ip_i$; the method attribution is in
Remark~\ref{M-rem:flag-attribution}.
\end{remark}

\section{The sharp bound for every four-letter alphabet}\label{M-sec:four}

\begin{theorem}[Theorem~\ref{M-thmA}]\label{M-thm:sharp-four-model}
Every ordered four-element real alphabet is order-preservingly $q$-equivalent
to a primitive integer alphabet $\{0,a,b,c\}$ with
\[
0<a<b<c\le q(q+1)\qquad(q\ge2).
\]
The bound is sharp for every $q$. For $q=1$ the sharp bound is three.
\end{theorem}
\begin{proof}
The short-relation rank is $0$, $1$ or $2$; rank three would force the
translated vector to be zero. Rank two was settled after
Lemma~\ref{M-lem:mass-minor} (Corollary~\ref{M-cor:max-rank}).
Rank one is Theorem~\ref{M-thm:corank-model} with $m=4$ and $H\le q$.
For rank zero we give an explicit model with no nonzero $q$-relation. Set
\[
 b=\frac{q(q+1)}2+1,\qquad
 d=\begin{cases}q,&q\text{ even},\\q+1,&q\text{ odd},\end{cases}
 \qquad T_q=\{0,1,b,b+d\}.
\]
Its diameter is at most $q(q+1)$: at $q=2$ it is six, and for $q\ge3$
use $d\le q+1$ and $q(q+1)/2\ge q+2$.

Suppose $(\delta_0,x,y,z)$ is a nonzero balanced relation on this model
of mass $M\le q$. Put $u=y+z$, so $x+bu+dz=0$.
If $u=0$, then $x=-dz$ and $M=(d+1)|z|>q$ unless the relation is zero.
Otherwise negate it so that $u>0$. Since $|x|\le q<b$, necessarily
$z=-w<0$, $w\ge1$, whence $y=u+w$ and $x=dw-bu$.
The positive and negative masses imply $M\ge u+w$ and $M\ge w-x$.
Thus $w\le q-u$, $x\ge w-q$, and
\[
 bu=dw-x\le(d-1)w+q\le dq-(d-1)u,
 \qquad u(b+d-1)\le dq.
\]
But $2(b+d-1)-dq=q(q+1-d)+2d>0$, so $u=1$.
Write $q=2r+\varepsilon$, $\varepsilon\in\{0,1\}$; then
$d=2r+2\varepsilon$ and $b=dr+r+1+\varepsilon$.
If $w\le r$, then $x\le-r-1-\varepsilon<0$ and
\[
 M=w-x=b-(d-1)w\ge b-(d-1)r=q+1.
\]
If $w\ge r+1$, then $x\ge r+\varepsilon-1\ge0$ and
\[
 M=1+w+x=(d+1)w+1-b\ge(d+1)(r+1)+1-b=q+1.
\]
Both contradict the mass bound. Hence $T_q$ models the rank-zero case.

For sharpness the $m=4$ family of Theorem~\ref{M-thm:corank-model} is
$\{0,\theta,1,q\}$. Its only primitive short relation, up to sign,
is $(q-1,0,-q,1)$. Every increasing integer model has $c=qb$ and
$b\ge q+1$, giving $c\ge q(q+1)$; equality is attained at
$\{0,1,q+1,q(q+1)\}$. Finally $q=1$ tests only individual-letter
equalities, and four distinct integers require and admit diameter three.
\end{proof}

At $q=3$, the increasing model of $\{0,\theta,1,3\}$ requires diameter
twelve. If the bijection may reverse the two middle elements, the image
tuple $(0,8,3,9)$ has diameter nine and the same primitive short
relation $(2,0,-3,1)$.

\begin{remark}[the three ranks separately]\label{M-rem:four-ranks}
The proof gives the maximum of $M_q(T)$ over each rank separately. For rank
two the maximum is $q^2$, attained by $\{0,q,q+1,q^2\}$
(Corollary~\ref{M-cor:max-rank}); this value and its sharpness are Shin's
\cite[Corollary~7.3, Lemma~3.1, Proposition~10.6]{M-Shin26}. For rank one the
maximum is $q(q+1)$, attained by the profile of $\{0,\theta,1,q\}$, as in
Theorem~\ref{M-thm:sharp-four-model}. For rank zero there is a single profile, and its minimum model diameter
is by definition $L_4(q)$. The set $T_q$ has diameter
$b+d=\binom{q+2}{2}+[q\text{ odd}]$, so the proof above shows
$L_4(q)\le\binom{q+2}{2}+[q\text{ odd}]$; Shin's Theorem~10.3
\cite{M-Shin26} proves equality, and $T_q$ is his optimal ruler $A_h^{*}$. Thus
\[
 L_4(q)=\frac{(q+1)(q+2)}2+[q\text{ odd}],
\]
which is used as a known formula in Sections~\ref{M-sec:cluster}
and~\ref{M-sec:rank-sensitive}. For $q=2,\ldots,5$ the three per-rank maxima
$L_4(q)$, $q(q+1)$, $q^2$ are $6,6,4$; $11,12,9$; $15,20,16$; $22,30,25$.
Direct enumeration of primitive integer four-sets of diameter at most
$q(q+1)$ confirms these per-rank maxima for $q=2,\ldots,5$.
For the cardinality alone, $\nu(h,4)=L_4(h)$ through $h=10$ and
$\nu(h,4)=(h-3)(h-1)$ from $h=11$ onward \cite{M-Sumset26}.
\end{remark}

\section{Rank-one models and the sharp five-letter bound}\label{M-sec:five}

The only difficult relation rank for five letters is one. We prove the
more precise bound $\mu q(q+1)$ first, where $\mu$ is the primitive
relation mass. The universal bound then follows by considering the four
possible ranks. This separates the geometric case reduction from the
integer constructions that implement it.

Throughout this section $q\ge2$. For a rank-one five-letter profile,
let $\delta$ be the primitive balanced generator of its complete
short-relation space and let $\mu=\mass(\delta)$. In translated
coordinates $x_0=0$, write
\[
 K=\{(x_1,\ldots,x_4):\delta\cdot(0,x_1,\ldots,x_4)=0,
                    \ 0\le x_1\le\cdots\le x_4\}.
\]
This ordered kernel cone has dimension three.

\subsection{Two insertion lemmas and an ordering face}
\begin{lemma}\label{M-lem:append-qc-plus-one}
Suppose $0=x_0<x_1<\cdots<x_{m-1}=c$ has no nonzero balanced relation of mass
at most $q$. Append $D=qc+1$. The enlarged alphabet still has no such
relation.
\end{lemma}
\begin{proof}
Suppose a short relation uses $D$ with nonzero coefficient. Reverse its sign
if necessary so that this coefficient is positive. Its positive weighted sum
is at least $D$, because all letters are nonnegative. Its negative weighted
sum is at most $qc$: its total negative mass is at most $q$ and every
negatively used letter is at most $c$. This contradicts $D>qc$. Therefore the
coefficient of $D$ vanishes, and the original relation-free property excludes
the relation.
\end{proof}

\begin{lemma}\label{M-lem:model-near-end}
Every four-letter rank-one $q$-profile of primitive relation mass
$\mu\le q$ has an integer model $S=\{0,A,B,C\}$ with $A\ge2$ and
$C\le\mu(q+1)$, without requiring $S$ to be primitive. Consequently a
five-letter rank-one profile with $\delta_1=0$ or $\delta_3=0$ has a
primitive model of diameter at most $\mu q(q+1)$.
\end{lemma}
\begin{proof}
In the two-dimensional ordered kernel cone of the four-letter core,
choose a primitive original ray $u$ with positive first coordinate; one
exists because the cone contains a strictly ordered point. Its height
is at most $\mu$. The $q$-chamber adjacent to it has another primitive
ray $v$ of height at most $\mu q$. If $v$ also has positive first coordinate,
use $u+v$. Otherwise $v$ lies on the supporting line $x_1=0$ and is an
original cone ray, so has height at most $\mu$; use $2u+v$. Its height
is at most $3\mu\le\mu(q+1)$. Both positive combinations are strictly
inside one full chamber and have first coordinate at least two. Do not
divide their gcd, which could shrink that coordinate.

If $\delta_1=0$, deleting that letter leaves exactly the original
rank-one core profile, because any core relation extends by a zero
coefficient to a relation of the five-letter alphabet, hence is a multiple
of $\delta$. Insert the free letter using
\[
 \{0,1,qA,qB,qC\}.
\]
In a short relation its coefficient $x$ at $1$ is divisible by $q$ and has
$|x|\le q$. If $x=q$, it consumes all positive mass; after division
by $q$, the negative side would express $1$ as a sum of $q$ nonnegative
letters of $S$. This is impossible because $A\ge2$. Sign reversal
excludes $x=-q$. Thus $x=0$ and only the core relations remain. The
new model is primitive because it contains $1$. Reflection and translation
give the case $\delta_3=0$, with the same diameter and increasing
correspondence.
\end{proof}

\begin{lemma}\label{M-lem:model-append-minus-one}
Let an integer alphabet in $[0,C]$ contain $0$ and $C$, with smallest
positive letter and last gap both at least two. Appending $qC-1$
creates no new balanced relation of mass at most $q$.
\end{lemma}
\begin{proof}
Orient a proposed relation so that the new letter $D=qC-1$ has positive
coefficient. Since $2D>qC$, that coefficient is one. Pad the negative
side with zeros to $q$ terms. The equality becomes
\[
 1=\sum_{\text{negative terms }t}(C-t)
    +\sum_{\substack{\text{positive terms }t\\t\ne D}}t.
\]
Every term is nonnegative and every nonzero term is at least two, an
impossibility. Existing relations are unaffected.
\end{proof}

Call a face $K\cap\{x_i=x_{i+1}\}$ of $K$ ($0\le i\le3$, with $x_0=0$) a
\emph{genuine facet} if it has dimension two. Suppose a genuine facet
$x_i=x_{i+1}$ collapses two adjacent letters; merge their coefficients in
$\delta$ and divide by the coefficient gcd $g$, and let the induced
primitive four-letter relation have mass $\mu'$. On that facet the first two
flag rays have heights at most $\mu'$ and $\mu'q$, while a ray outside it in
an adjacent full chamber has height at most $\mu q^2$. To see the first two
heights, let $\rho_i$ be the ordering row of the facet ($\rho_0=e_1$,
$\rho_i=e_{i+1}-e_i$ for $i\ge1$) and replace the deleted row
$(\delta_1,\ldots,\delta_4)$ of $\delta$ on this facet by
$\bigl((\delta_1,\ldots,\delta_4)-\delta_{i+1}\rho_i\bigr)/g$, the deleted
row of the merged primitive relation, which has mass $\mu'$ and vanishes in
coordinate $i+1$. Each of the two rays is cut out by $\delta$, $\rho_i$ and
one further active row, an ordering row for the first ray and an ordering or
$q$-relation row for the second (proof of Theorem~\ref{M-thm:flag-model}).
The row operation and the division change neither the kernel line nor its
primitive generator, and the identification of the two coincident
coordinates preserves heights; so Lemma~\ref{M-lem:mass-minor}, applied with
the merged row in place of $\delta$, gives heights at most $\mu'$ (one
ordering row) and $\mu'q$ (one further $q$-relation row). Their sum gives
\begin{equation}\label{M-eq:model-mass-drop}
 \operatorname{diameter}\le\mu q^2+\mu'(q+1).
\end{equation}
This is a relative weak-face cofactor argument. Forced equalities are
allowed on the face; functionals identically zero on its span are ignored
there and become strict after adding the last ray. No theorem requiring
four strictly distinct collapsed letters is being applied. If adjacent
coefficients have opposite signs, their merger reduces mass; a gcd can
also reduce primitive mass.

\subsection{A complete reduction by support and signs}

\begin{lemma}[Reduction of rank-one profiles]\label{M-lem:rank-one-support}
Let $\delta\in\Z^5$ be a primitive balanced relation, of mass $\mu$, that
vanishes on a strictly increasing alphabet. At least one of the following
holds, after reflection and reversal of the relation's sign when needed:
\begin{enumerate}
\item $\delta_1=0$ or $\delta_3=0$;
\item an ordering facet merges adjacent coefficients and gives a primitive
      relation of mass strictly less than $\mu$;
\item an ordering facet forces two gaps to be zero and leaves the other
      two gaps free and nonnegative;
\item the support is $\{0,1,3\}$, with relation
      $\mu t_1=s t_3+(\mu-s)t_0$, $1\le s<\mu$, $\gcd(s,\mu)=1$;
\item $\delta=(R,u,0,-\mu,N)$, where $R,u,N>0$,
      $\mu=R+u+N$ and $\gcd(N,\mu)=1$.
\end{enumerate}
\end{lemma}
\begin{proof}
Assume $\delta_1\delta_3\ne0$. Put
\[
 A_j=\sum_{i=j}^4\delta_i\quad(1\le j\le4),\qquad A_0=A_5=0.
\]
Here $\delta_i=A_i-A_{i+1}$, and the positive gaps $g_j$ satisfy
$\sum A_jg_j=0$, so some tails are positive and some negative.
At a strict extremum $A_j$ the coefficients $\delta_{j-1},\delta_j$
have opposite signs. If removing $A_j$ leaves both signs, setting
$g_j=0$ gives a two-dimensional facet on which the merged coefficient
has smaller mass, as in (2). If only one nonzero tail $A_k$ remains,
the same face also forces $g_k=0$, giving (3).
The exceptional patterns can be read from the tails as follows.
\begin{center}\small
\begin{tabular}{@{}ll@{}}
\toprule
Pattern after reflection or sign change & Cases below\\
\midrule
$\delta_1\delta_3=0$ & (1)\\
All five coefficients nonzero & (2)\\
$\delta_4=0$, with $\delta_1\delta_3\ne0$ & (2), (3), or (4)\\
Only $\delta_2=0$, tails $(A_1,A_2,A_2,A_4)$ & (2), (3), or (5)\\
\bottomrule
\end{tabular}
\end{center}

If all five coefficients of $\delta$ are nonzero, the endpoints $A_1,A_4$
are nonzero and adjacent tail sums are distinct. Some sign repeats:
otherwise $A_2=A_3=0$, contrary to $\delta_2\ne0$. A global extremum
of a repeated sign is strict at its neighbors, and its deletion leaves
both signs. This gives (2).

Suppose an outer coefficient is zero, say $\delta_4=0$. Then
$A=(A_1,A_2,A_3,0)$ and $A_3\ne0$. If $\delta_2=0$, feasibility
forces support $\{0,1,3\}$; a relation on only two distinct letters
cannot vanish. Primitivity and the order give exactly (4). Otherwise
the first three adjacent tail sums differ. A repeated sign gives (2)
by its strict global extremum. With no repeated sign there is one
positive tail, one negative tail and two zeros; deleting a nonzero
extremum gives (3). Reflection handles $\delta_0=0$.

It remains that only the middle coefficient is zero. Now
$A=(A_1,A_2,A_2,A_4)$ has nonzero endpoints, each different from $A_2$.
If $A_2=0$, (3) applies. If the endpoints have the same sign, an
endpoint extremum gives (2). Otherwise sign reversal and reflection
reduce to $A=(-R,-P,-P,N)$ with $R,P,N>0$ and $R\ne P$.
For $R>P$, deleting the first tail gives (2). For $R<P$, set $u=P-R$;
then $\delta=(R,u,0,-(P+N),N)$ and $\mu=P+N$.
If $g=\gcd(P,N)>1$, the genuine facet $g_1=0$ has merged row
$(P,0,-\mu,N)$ and primitive mass $\mu/g<\mu$, again (2).
Otherwise $\gcd(N,\mu)=1$ and (5) holds. These cases exhaust the
zero patterns and signs, independently of the observation order $q$.
\end{proof}

\subsection{The bound in terms of relation mass}
\begin{proposition}\label{M-prop:rank-one-mass}
Let $q\ge2$ and let $T$ be an ordered five-element real alphabet whose
complete space of vanishing $q$-relations has rank one, with primitive
balanced generator $\delta$ of mass $\mu$. Then $2\le\mu\le q$, and $T$ has
a primitive order-preserving $q$-model of diameter at most
\begin{equation}\label{M-eq:mass-sensitive}
 \mu q(q+1).
\end{equation}
\end{proposition}

Every integer multiple $k\delta$ with $|k|\mu\le q$ belongs to the profile
of $T$ and must be preserved; the proof allows all such multiples and never
assumes $\mu=q$, nor that only the two primitive signs occur at order $q$.
The mass is at least two because a balanced relation of mass one would
identify two distinct letters.

\begin{proof}
Work in the three-dimensional ordered kernel cone of $\delta$, subdivided
by all $q$-relation hyperplanes. Lemma~\ref{M-lem:rank-one-support} gives
five cases. We establish the bound in each case, retaining every multiple
$k\delta$ with $|k|\mu\le q$ and excluding every other short relation.

\smallskip\noindent\emph{Cases (2) and (3): the two face estimates.}
If collapsing two adjacent letters defines a genuine facet, merging the two
coefficients of $\delta$ and dividing by the gcd gives a primitive
four-letter row of mass $\mu'$, and \eqref{M-eq:model-mass-drop} holds with
the present $\mu$: on the facet the deleted row of $\delta$ is replaced by
the merged row, which changes neither the kernel lines of the first two flag
rays nor their heights, so Lemma~\ref{M-lem:mass-minor} bounds these heights
by $\mu'$ and $\mu'q$ as in \eqref{M-eq:model-mass-drop},
while the third ray, outside the facet in an adjacent full chamber, is cut
out by the original row of mass $\mu$ and two further rows of mass at most
$q$, so has height at most $\mu q^2$. If $\mu'<\mu$, then
\begin{equation}\label{M-eq:mass-drop-budget}
 \mu q^2+(\mu-1)(q+1)=\mu q(q+1)-(q-\mu+1)\le\mu q(q+1)-1 .
\end{equation}
In case (3), let the free gaps be $g_a,g_b$, $a<b$.
Set $g_a=1$, $g_b=q$ and all other gaps to zero. The resulting repeated
letters take the three values $0,1,q+1$. No nonzero $q$-relation on
these three values vanishes: reduction modulo $q+1$ forces the
coefficient of $1$ to be zero, and the remaining coefficients then vanish.
Consequently this vector $z$, of height $q+1$, lies in the relative
interior of a full short-relation chamber on the face. Extend this
chamber by Lemma~\ref{M-lem:chamber-extension} and choose an extreme ray
$w$ outside the face with height at most $\mu q^2$.
A chamber functional vanishing at $z+w$ would vanish on $w$ and on
the whole span of the face, hence on the original kernel. Thus $z+w$
is strictly ordered and has exactly the required zero profile.
After gcd division its height is at most
\begin{equation}\label{M-eq:two-zero-budget}
 \mu q^2+q+1\le\mu q(q+1)-1,
\end{equation}
since $(\mu-1)q\ge2$.

\smallskip\noindent\emph{Case (1): a zero coefficient next to an endpoint.}
If $\delta_1=0$ or $\delta_3=0$, Lemma~\ref{M-lem:model-near-end}, which was
proved for an arbitrary primitive mass $\mu\le q$, gives a four-letter core
$\{0,A,B,C\}$ with $A\ge2$ and $C\le\mu(q+1)$, and the insertion
$\{0,1,qA,qB,qC\}$ is a primitive model of diameter at most $\mu q(q+1)$.

\smallskip\noindent\emph{The remaining supports.}
Beyond these three situations the exhaustion leaves the support
$\{0,1,3\}$, its reflection, and, with $u=P-R$, the coprime middle-free
shape
\begin{equation}\label{M-eq:final-shape}
 u=P-R>0,\quad R>0,\quad N>0,\quad \mu=R+u+N\le q,\quad \gcd(N,\mu)=1 .
\end{equation}
These are treated in the last two paragraphs.

\smallskip\noindent\emph{Case (4): the outer-free three-anchor support.}
For the support $\{0,1,3\}$ the primitive row reads
$\mu t_1=s\,t_3+(\mu-s)\,t_0$ with $1\le s\le\mu-1$ and $\gcd(s,\mu)=1$.
Put $B=q+1$ and $C=\mu B$. The core $\{0,sB,sB+1,C\}$ has exactly the
required relation line at order $q$: modulo $B$ the coefficient of $sB+1$
vanishes, its absolute value being at most $q<B$, and the remaining
equation $y\,sB+w\,\mu B=0$ has by coprimality exactly the solutions
$(y,w)=k(\mu,-s)$, the integer multiples of the primitive row, including
every multiple permitted at order $q$. Its first positive letter is
$sB\ge q+1\ge3$ and its last gap is $(\mu-s)B-1\ge q\ge2$, so
Lemma~\ref{M-lem:model-append-minus-one} applies: appending $D=qC-1$ gives
the increasing primitive model $\{0,sB,sB+1,\mu B,q\mu B-1\}$, of diameter
\begin{equation}\label{M-eq:three-anchor-diameter}
 D=\mu q(q+1)-1 .
\end{equation}
Reflection supplies the opposite outer-free support.

\smallskip\noindent\emph{Case (5): the coprime middle-free case.}
Assume \eqref{M-eq:final-shape}. Choose $C_0\in\{1,\ldots,\mu-1\}$ with
$NC_0\equiv-u\pmod\mu$ and define
\begin{equation}\label{M-eq:middle-free-anchors}
 B_0=\frac{NC_0+u}{\mu},\qquad B=Nq+B_0,\qquad C=\mu q+C_0 .
\end{equation}
The residue $C_0$ exists and is nonzero because $0<u<\mu$ and
$\gcd(N,\mu)=1$; from $0<NC_0+u<(N+1)\mu$ we get $1\le B_0\le N$. Moreover
$\mu(C_0-B_0)=(\mu-N)C_0-u>0$, since $\mu-N=R+u>u$; hence $B_0<C_0$ and,
using $N\le\mu-1$, which holds since $R,u\ge1$, $B<C$. Also $B>q$,
$C/\mu>q$, and
\begin{equation}\label{M-eq:middle-free-identity}
 \mu B-NC=u .
\end{equation}
We check all $q$-relations on the anchors $S=\{0,1,B,C\}$.
Write such a relation as $x+yB+zC=0$. If $y,z$ have equal weak signs
and are not both zero, then $|x|\ge B>q$. Otherwise reverse the sign
so that $y\ge1$ and $z=-t\le-1$, and put $E=Ny-\mu t$.
Equation~\eqref{M-eq:middle-free-identity} gives $x=-(CE+uy)/\mu$.
For $E\ge1$, $|x|\ge C/\mu>q$. For $E\le-1$, $C>uq$
gives $x>0$ and
\[
 x+y=C|E|/\mu+(1-u/\mu)y>q.
\]
Both contradict the mass bound. For $E=0$, coprimality gives
$y=\mu k$, $t=Nk$, $x=-uk$ and balancing coefficient $-Rk$.
These are precisely the multiples of the prescribed anchor relation;
all multiples with $\mu|k|\le q$ are allowed.

Choose $X=q+1$ if $N\ge2$. A $q$-term anchor sum is either at most
$q$, using only $0,1$, or at least $B\ge2q+1$, so it cannot be $X$.
If $N=1$, then $B_0=1$, $C_0=\mu-u=R+1\ge2$, and
$B=q+1$, $C\ge\mu q+2\ge3q+2$. Choose $X=2q+1$.
An anchor sum using no $C$ and zero, one, or at least two copies of
$B$ is respectively at most $q$, at most $2q$, or at least $2q+2$;
any use of $C$ is larger still. Again $X$ is absent.
Now form
\begin{equation}\label{M-eq:middle-free-model}
 \{0,\ q,\ X,\ qB,\ qC\}.
\end{equation}
Here $q<X<qB<qC$, the congruence $X\equiv1\pmod q$ forces the
coefficient at $X$ in a $q$-relation to be $0$ or $\pm q$, and the nonzero
case would express $X$ as a $q$-term sum of $S$; so the profile of
\eqref{M-eq:middle-free-model} is exactly the line of $\delta$ with all its
permitted multiples, the model is primitive because $\gcd(q,X)=1$, and its
diameter is
\begin{equation}\label{M-eq:middle-free-diameter}
 qC=\mu q^2+qC_0\le\mu q(q+1)-q .
\end{equation}
The five budgets are all at most $\mu q(q+1)$, which proves
\eqref{M-eq:mass-sensitive}.
\end{proof}

\begin{remark}[attribution]\label{M-rem:mass-attribution}
Shin's filtered estimate \cite[Theorem~7.2 and Remark~7.5]{M-Shin26}
gives $3\mu q^2$ at five letters and rank one for the complete ordering
of the sums. Here the first relation degree is $\mu$. Preserving the
equalities gives the smaller bound $\mu q(q+1)$.
\end{remark}

\subsection{The universal five-letter bound}
\begin{theorem}[Theorem~\ref{M-thmB}]\label{M-thm:sharp-five-model}
Every ordered five-element real alphabet has an order-preserving
$q$-equivalent primitive integer model of diameter at most $q^2(q+1)$,
for every integer $q\ge2$. This bound is sharp for every $q$.
\end{theorem}

\begin{proof}[Proof of Theorem~\ref{M-thm:sharp-five-model}]
The short-relation rank is at most three. Rank three has a unique rational
kernel direction with primitive height at most $q^3$, by
Lemma~\ref{M-lem:mass-minor}. Rank two is Theorem~\ref{M-thm:corank-model}.
For rank zero, append $qc+1$ to the relation-free four-letter alphabet
$\{0,1,b,c\}$ constructed in Theorem~\ref{M-thm:sharp-four-model}.
Any short relation using the new letter has, after sign reversal, positive
sum at least $qc+1$ and negative sum at most $qc$, a contradiction.
For $q\ge3$, the displayed four-letter construction satisfies
\[
 qc+1\le\frac{q^3+3q^2+4q+2}{2}\le q^2(q+1),
\]
since $q^3-q^2-4q-2>0$. For $q=2$ use $\{0,1,4,9,11\}$; its fifteen
unordered two-term sums are distinct.

For rank one, Proposition~\ref{M-prop:rank-one-mass} gives
diameter at most $\mu q(q+1)\le q^2(q+1)$.
The alphabet $\{0,\theta,1,q,q^2\}$ with irrational $0<\theta<1$
forces diameter at least $q^2(q+1)$ by Theorem~\ref{M-thm:corank-model},
which also supplies its equality model. This proves sharpness.
\end{proof}

\begin{corollary}[relation-free five-letter models]\label{M-cor:five-rank-zero}
For every $q\ge2$ there is a five-element integer alphabet with no nonzero
$q$-relation and diameter at most $q^2(q+1)$; hence $L_5(q)\le q^2(q+1)$.
Explicitly, for $q\ge3$ the alphabet $\{0,1,b,c,qc+1\}$, with $b$ and $c=b+d$ as in
the proof of Theorem~\ref{M-thm:sharp-four-model}, is primitive, increasing,
relation-free, and of diameter at most $(q^3+3q^2+4q+2)/2$; for $q=2$ the
alphabet $\{0,1,4,9,11\}$ has diameter $11$.
\end{corollary}
\begin{proof}
This is the rank-zero construction in the proof above, with
Lemma~\ref{M-lem:append-qc-plus-one}; the presence of $1$ ensures
primitivity. Since $d\le q+1$,
$qc+1\le q\bigl(q(q+1)/2+q+2\bigr)+1=(q^3+3q^2+4q+2)/2$, and
$q^2(q+1)-(q^3+3q^2+4q+2)/2=(q^3-q^2-4q-2)/2>0$ for $q\ge3$, because
$q^3-q^2-4q-2=q^2(q-1)-4q-2\ge2q^2-4q-2>0$.
\end{proof}

For $q=2$, equality of two distinct positive differences is equivalent
to a nontrivial equality of two pair sums. Thus $L_k(2)$ is the length
of an optimal $k$-mark Golomb ruler. In particular $L_5(2)=11$,
attained by $\{0,1,4,9,11\}$ \cite[Table~2.2]{M-Dim02}. Finite
enumeration gives $L_5(3)=23$, attained by $\{0,1,15,18,23\}$, and
$L_5(4)=41$, attained by $\{0,1,24,37,41\}$. The construction in the
corollary has diameters $11,34,61$ at these three orders.
Shin's five-point superincreasing set \cite[Corollary~7.4]{M-Shin26} has
height $q^3+q^2+q+1$.

\section{Face lifting and the general bounds}\label{M-sec:lifting}

The universal radii $H_m(q)$ are finite by Theorem~\ref{M-thm:flag-model} and
nondecreasing in $m$ by Lemma~\ref{M-lem:elementary}. A known exact bound in a
smaller alphabet propagates upward through the facets of the ordered kernel
cone.

\begin{theorem}[face-lifting recurrence]\label{M-thm:lifting}
For $m\ge3$ and $q\ge2$,
\begin{equation}\label{M-eq:recurrence}
 H_m(q)\le q^{m-2}+H_{m-1}(q).
\end{equation}
\end{theorem}
\begin{proof}
We first model the coincident letters on an ordering facet, then lift
that model to a compatible chamber and check its order, relations and diameter.
Given an $m$-letter real alphabet, let $V$ be its complete $q$-relation span,
and let $K=K_0$ be the weak-order cone in its translated real kernel $L$, as
in Section~\ref{M-sec:flag}. It has full relative dimension $d$ in $L$ and
contains a strictly ordered point.

If $d=1$, a primitive ray generator gives diameter at most $q^{m-2}$ directly
by the balanced-row cofactor bound (Corollary~\ref{M-cor:max-rank}), so
\eqref{M-eq:recurrence} follows. Suppose $d\ge2$ and take an original ordering
facet $F$ of $K$, that is, a facet of the form $K\cap\{x_1=0\}$ or
$K\cap\{x_i=x_{i+1}\}$. Its nonzero relative interior has a fixed pattern of
coincident adjacent letters, hence $k$ distinct value blocks, for some
$2\le k\le m-1$: there are at least two blocks, because a single common value
with the first translated coordinate fixed at zero would give only the zero
vector, and there are at most $m-1$ blocks, because a proper ordering facet
has at least one active gap equality. Several equalities may be forced
simultaneously on $F$, so $k$ may be smaller than $m-1$.

\smallskip\noindent\emph{Model on the face.}
Choose $p$ in that relative interior avoiding every $q$-hyperplane that does
not contain $\spn F$. This is possible because there are finitely many
proper hyperplanes. Let $0=\pi_1<\cdots<\pi_k$ be the distinct values among
$0,p_1,\ldots,p_{m-1}$ and $\beta(i)$ the block of letter $i$, so that
$p_i=\pi_{\beta(i)}$. Choose an order-preserving $q$-model
$0=\pi'_1<\cdots<\pi'_k$ of $(\pi_1,\ldots,\pi_k)$ of diameter at most
$H_k(q)$ and set $u_i=\pi'_{\beta(i)}$. The resulting integer vector $u$
then has $\height(u)=\diam\le H_k(q)\le H_{m-1}(q)$.

It lies in $F$. Indeed every original $q$-relation collapses, by summing its
coefficients within each block, to a balanced relation on the $k$ blocks of
mass no larger than $q$, since summing within blocks only permits
cancellation; its value at the repeated vector equals the value of the
collapsed relation at the block values. The collapsed relation vanishes at
the block values of $p$, so the block model preserves it, and the original
relation vanishes on $u$; as these relations span $V$, $u\in L$. All the
required equalities between adjacent letters are preserved by repetition, and
all other ordering gaps are positive. Likewise every $q$-form not identically
zero on $F$ remains nonzero at $u$: it was nonzero at $p$ and collapses to a
relation of mass at most $q$ on the value blocks, whose nonvanishing the
block model preserves. Thus $u$ is in the relative interior of a full
short-hyperplane chamber $C_F$ on $F$, even when $F$ has multiple forced
equalities among the original $m$ coordinates. Its nonzero comparison signs
need not match those of $p$; what follows uses the chamber that actually
contains $u$.

\smallskip\noindent\emph{Lift to a chamber.}
Lemma~\ref{M-lem:chamber-extension} supplies an adjacent $d$-dimensional
chamber $C$ inside $K$ with $C\cap F=C_F$. Choose a primitive integer extreme
ray $v$ of $C$ outside $\spn F$; one exists because $C$ spans $L$. It is
determined by $m-2$ independent original, short-relation and ordering rows,
each of mass at most $q$: an extreme ray of the $d$-dimensional pointed
cone $C$ lies on $d-1$ independent active hyperplanes of its defining list,
by the perturbation argument in the proof of Theorem~\ref{M-thm:flag-model},
and these together with the $r$ rows spanning $V$ are $m-2$ independent
rows. So its height is at most $q^{m-2}$ by Lemma~\ref{M-lem:mass-minor}.

\smallskip\noindent\emph{Order, relations and diameter.}
The vector $u+v$ is in the relative interior of $C$. A supporting functional
of $C$ that vanished on $u$ must vanish on the entire span of $F$, and $v$
makes it strict; every other facet functional is already strict on $u$.
Equivalently, $u$ supplies all directions within $F$ and $v$ leaves that
facet. Hence $u+v$ is strictly ordered and lies on no $q$-hyperplane outside
$V$, while all original relations vanish on it. Gcd division gives the
desired primitive model, of height at most $H_{m-1}(q)+q^{m-2}$. This proves
\eqref{M-eq:recurrence}.
\end{proof}

\begin{theorem}[Theorem~\ref{M-thmF}]\label{M-thm:general-bounds}
For $m\ge3$ and $q\ge2$,
\begin{equation}\label{M-eq:two-sided}
 q^{m-2}+q^{m-3}\le H_m(q)\le q^{m-2}+H_{m-1}(q),
 \qquad\text{and also}\qquad H_m(q)\le\sum_{j=0}^{m-2}q^j .
\end{equation}
Consequently
\begin{equation}\label{M-eq:m-ge-5}
 q^{m-2}+q^{m-3}\le H_m(q)\le q^2+q^3+\cdots+q^{m-2}\qquad(m\ge5),
\end{equation}
for each fixed $m\ge4$
\begin{equation}\label{M-eq:asymptotic}
 H_m(q)=q^{m-2}+q^{m-3}+O_m(q^{m-4}),
\end{equation}
the lower endpoint of \eqref{M-eq:two-sided} is exact for $m=3,4,5$, namely
$H_3(q)=q+1$, $H_4(q)=q(q+1)$, $H_5(q)=q^2(q+1)$, and
\begin{equation}\label{M-eq:six-bracket}
 q^4+q^3\le H_6(q)\le q^4+q^3+q^2 .
\end{equation}
\end{theorem}
\begin{proof}
Theorem~\ref{M-thm:corank-model} gives the lower family: its $m$-letter
member $\{0,\theta\}\cup\{q^j:0\le j\le m-3\}$ has rank $m-3$ and forces every
increasing integer model to have diameter at least $q^{m-3}(q+1)$. The first
upper bound is Theorem~\ref{M-thm:lifting}. For a complete relation space of
rank $r$, Theorem~\ref{M-thm:flag-model} and $H\le q^r$ give the bound
$\sum_{j=r}^{m-2}q^j$ of \eqref{M-eq:rank-sensitive-sum}, at most the displayed
geometric sum. Iterating \eqref{M-eq:recurrence} from the independently proved
value $H_5(q)=q^3+q^2$ of Theorem~\ref{M-thm:sharp-five-model} gives
\eqref{M-eq:m-ge-5}. The difference between the upper and lower bounds in
\eqref{M-eq:m-ge-5} is $\sum_{j=2}^{m-4}q^j$, empty for $m=5$, and for
the geometric upper bound in \eqref{M-eq:two-sided} it is $\sum_{j=0}^{m-4}q^j$; either proves the fixed-$m$
asymptotic statement \eqref{M-eq:asymptotic}. For $m=3$ the bounds coincide,
since $H_2(q)=1$ gives $H_3(q)\le q+1$. Theorems~\ref{M-thm:sharp-four-model}
and~\ref{M-thm:sharp-five-model} give the other two exact cases, and $m=6$ in
\eqref{M-eq:m-ge-5} is \eqref{M-eq:six-bracket}.
\end{proof}

At order two, the recurrence gives the finite search ranges in
Section~\ref{M-sec:six}.

\begin{corollary}\label{M-cor:seven-two}
$H_7(2)=48$.
\end{corollary}
\begin{proof}
By \eqref{M-eq:recurrence} and Theorem~\ref{M-thm:six-two}, $H_7(2)\le 2^5+24=56$.
The complete comparison of equality types through diameter $56$ in
Section~\ref{M-sec:all-types} attains its last new type at diameter $48$.
\end{proof}

\begin{corollary}[order-preserving Freiman copies]\label{M-cor:label}
Let $h\ge2$ and $k\ge3$. Every $k$-element set $A$ of real numbers has an
order-preserving Freiman $h$-isomorphic copy $B\subset\Z$ with
\[
 \diam B\le H_k(h)\le 1+h+h^2+\cdots+h^{k-2},
\]
and $\diam B\le h^2+\cdots+h^{k-2}$ if $k\ge5$. In particular
$|hB|=|hA|$, since the number of distinct $h$-term sums is determined by the
$h$-zero profile.
\end{corollary}
\begin{proof}
Remark~\ref{M-rem:freiman} and Theorem~\ref{M-thm:general-bounds}.
\end{proof}

For $h,k\ge3$, the proof of Nathanson's Theorem~8, using Lemma~1
\cite{M-Nat26}, yields a Freiman $h$-isomorphic integer copy of each
$k$-element integer set with diameter below $4(8h)^{k-1}$, without
preserving order. The theorem itself states the corresponding bound for
$N(h,k)$; see \S\ref{M-sec:intro-label} for this distinction and the
label-level bounds of \cite{M-Shin26}. The geometric sum in
\eqref{M-eq:two-sided} has a familiar counterpart for the empty profile:
Nathanson \cite{M-Nathanson25} proves this type of bound for the greedy $B_q$
sequence, whose unordered $q$-term sums are distinct, and determines
its third positive element as $q^2+q+1$. Here the flag bound applies to
every prescribed profile.

\section{Six and seven letters at order two}\label{M-sec:six}

\begin{theorem}[Six-letter case of Theorem~\ref{M-thmC}]\label{M-thm:six-two}
Every ordered six-element real alphabet has an order-preserving
$2$-equivalent primitive integer model of diameter at most $24$, and this
bound is sharp. Thus $H_6(2)=24$.
\end{theorem}

Translate the first letter to zero and let $r$ be the rank of the complete
short-relation space in five translated coordinates. A nonzero ordered point
in its kernel forces $r\le4$. The table groups all six-letter types by rank and gives the largest
minimum diameter within each rank.
\begin{center}
\begin{tabular}{@{}rrl@{}}
\toprule
rank $r$ & feasible profiles & largest minimum diameter\\
\midrule
$0$ & $1$ & $17$\\
$1$ & $35$ & $22$\\
$2$ & $306$ & $24$\\
$3$ & $697$ & $24$\\
$4$ & $266$ & $16$\\
\bottomrule
\end{tabular}
\end{center}
Section~\ref{M-sec:all-types} enumerates all $1{,}305$ profiles. The rank-three bound is $(2+1)2^3=24$, and rank four has a unique rational
kernel direction of primitive diameter at most $2^4=16$. The corank-two
family $\{0,\theta,1,2,4,8\}$, with
irrational $0<\theta<1$, forces diameter at least $24$, and its model
$\{0,1,3,6,12,24\}$ attains that value. The finite low-rank classification
below supplies the only remaining universal upper-bound cases.

\subsection{All order-two hyperplanes}
After canceling common terms in a two-term-sum equality, every nontrivial
equality on strictly increasing letters is either
\[
 t_i+t_k=2t_j\quad(i<j<k)\qquad\text{or}\qquad
 t_i+t_l=t_j+t_k\quad(i<j<k<l).
\]
There are $\binom63=20$ of the first kind and $\binom64=15$ of the second,
for $35$ distinct hyperplanes. The other four-distinct-letter pairings are
termwise unequal, and an equality involving only two distinct letters would
contradict strict order. These cases exhaust all balanced relations of mass
at most two. The full list, in the order used by the data files, is given in
the technical supplement.

For positive gaps $g_j=t_j-t_{j-1}$, $j=1,\ldots,5$, a balanced relation
$\delta$ has gap coefficients
\[
 A_j=\sum_{i=j}^5\delta_i,\qquad \delta\cdot T=\sum_jA_jg_j .
\]
The cumulative-sum coordinate map is invertible over the integers, so it
preserves relation ranks and row-space closure.

\subsection{The finite classification}
For six letters the low-rank classification also identifies which
relation patterns attain the largest diameter within each rank.

\Needspace{8\baselineskip}
\begin{proposition}\label{M-prop:six-classification}
Let $q=2$ and $m=6$.
\begin{enumerate}
\item Among the $1+35+\binom{35}2=631$ subsets of at most two of the $35$
 hyperplane rows, the rational row-space closures (the set of all $35$ rows
 lying in the span of the subset) give exactly $487$ distinct flats. The
 $2$-zero profile of every ordered real six-letter alphabet of complete
 short-relation rank at most two is one of these $487$ flats.
\item Exactly $342$ of the flats, namely $1$, $35$ and $306$ of ranks $0$,
 $1$ and $2$, contain a strictly ordered real alphabet. For each of the
 other $145$ flats, all of rank two, there is an explicit nonzero nonnegative
 integer vector in gap coordinates that is an integer linear combination of
 the two basis gap rows; such a vector has positive inner product with every
 strictly positive gap vector, whereas the basis equations would make that
 product zero, so the flat contains no strictly ordered real alphabet.
\item Each of the $342$ feasible flats has a primitive strictly increasing
 integer model beginning at $0$ of diameter at most $24$ with exactly that
 zero profile. The minimum diameters are exact, and their maxima are $17$,
 $22$ and $24$ in ranks $0$, $1$ and $2$.
\end{enumerate}
\end{proposition}

\begin{proof}
A real profile of rank $r\le2$ has a basis of $r$ rows among the $35$
hyperplanes. Its full zero mask is exactly the closure of that basis among
all $35$ rows: every row in the span vanishes, and every vanishing row
belongs to the full span by definition. Different bases with the same
closure have the same rational row space, so deduplicating by the closure is
valid. The $631$ candidate bases give $487$ closures. A test for
membership of a third row uses minors of size at most three with
entries bounded by two. Their absolute values are at most $3!2^3=48$,
so the closure can also be computed over the field of $101$ elements.

For $342$ closures the certificate gives a strictly increasing integer
model with precisely the indicated equalities. Each of the other $145$
closures has a nonzero nonnegative integer vector in the span of its
basis rows in gap coordinates. The scalar product of this vector with
a strictly positive gap vector is positive, while the basis equations
require it to vanish. These certificates decide feasibility in the
whole positive cone.

Choose five entries from $\{1,\ldots,24\}$ and adjoin zero. Evaluating the
$35$ equations on all $\binom{24}{5}=42{,}504$ such tuples records the
least diameter attaining each of the feasible masks. The maxima by rank
are $17$, $22$ and $24$. A minimum tuple is primitive, since dividing
by a common factor would preserve its mask. The supplement lists the
row conventions and the distribution of these minima.
\end{proof}

The unique rank-zero profile has minimum model $\{0,1,4,10,12,17\}$. The two
rank-one profiles with minimum diameter $22$ are
\begin{center}
\begin{tabular}{@{}ll@{}}
\toprule
primitive relation & minimum model\\
\midrule
$(1,-2,0,0,0,1)$ & $\{0,11,12,15,20,22\}$\\
$(1,0,0,0,-2,1)$ & $\{0,1,4,9,11,22\}$\\
\bottomrule
\end{tabular}
\end{center}
There are four rank-two profiles with minimum diameter $24$; their models
are
\[
 \{0,6,7,10,12,24\},\quad
 \{0,12,18,19,22,24\},\quad
 \{0,12,13,16,18,24\},\quad
 \{0,1,4,6,12,24\}.
\]
Corollary~\ref{M-cor:six-explanations} explains the maxima $22$ and
$24$ by constructions that also apply at higher orders.

\begin{proof}[Proof of Theorem~\ref{M-thm:six-two}]
Let $T$ be an ordered six-element real alphabet of complete short-relation
rank $r$ at $q=2$. If $r=4$, Corollary~\ref{M-cor:max-rank} gives a primitive
model of diameter at most $2^4=16$. If $r=3$, Theorem~\ref{M-thm:corank-model}
gives one of diameter at most $(2+1)2^3=24$. If $r\le2$, the profile of $T$
is one of the $342$ feasible flats of
Proposition~\ref{M-prop:six-classification}, which has a primitive model of
diameter at most $24$. This proves the upper bound. For the lower bound,
Theorem~\ref{M-thm:corank-model} with $m=6$ and $q=2$ shows that every
order-preserving $2$-model of $\{0,\theta,1,2,4,8\}$ has diameter at least
$2^3(2+1)=24$: the three independent mass-two power relations force any
increasing integer $2$-model to be $\{0,a,b,2b,4b,8b\}$, and $b\le2$ would
create the extra relation $(a-b,b,-a,0,0,0)$ of mass $b$, contradicting
irrationality; hence $b\ge3$. The model $\{0,1,3,6,12,24\}$ attains equality,
since modulo $3$ every short relation forces the coefficient of $1$ to vanish
and leaves precisely the original power relations.
\end{proof}

\subsection{All pair-sum types through the recurrence bound}\label{M-sec:all-types}
For $m=6,7$, let $\mathcal P_m(D)$ be the set of equality patterns of
the tuples $0=a_0<a_1<\cdots<a_{m-1}\le D$. After cancelling terms,
every nontrivial pair-sum equality on an increasing tuple has one of
the two forms
\[
 a_i+a_k=2a_j\quad(i<j<k),\qquad
 a_i+a_l=a_j+a_k\quad(i<j<k<l).
\]
The other pairings of four distinct indices have unequal sides by
termwise comparison. Thus a bit for each triple and each quadruple
encodes the entire weak type. There are $35$ bits for $m=6$ and $70$
for $m=7$.

\begin{center}
\begin{tabular}{@{}rrrr@{}}
\toprule
$m$ & $D$ & tuples tested & $|\mathcal P_m(D)|$\\
\midrule
6 & 23 & 33,649 & 1,289\\
6 & 24 & 42,504 & 1,305\\
6 & 28 & 98,280 & 1,305\\
7 & 47 & 10,737,573 & 36,475\\
7 & 48 & 12,271,512 & 36,519\\
7 & 56 & 32,468,436 & 36,519\\
\bottomrule
\end{tabular}
\end{center}

For each diameter $d$ in increasing order, choose $m-2$ interior
entries from $\{1,\ldots,d-1\}$, adjoin $0,d$, and insert the resulting
bit vector in a set. After the last diameter $D$ the set equals
$\mathcal P_m(D)$ and the cumulative tuple count is $\binom{D}{m-1}$.
The ancillary program \texttt{models\_q2.cpp} computes this table using
integer addition, comparison and exact bit-vector equality. For the six-letter
table at the start of this section we group profiles by the rational rank
of their vanishing relation rows. The first diameter at which a profile
appears is its minimum. The program \texttt{six\_letter\_ranks.py}
computes the ranks by rational elimination and takes the largest
minimum diameter in each group.

By \eqref{M-eq:recurrence}, $H_6(2)\le 2^4+H_5(2)=28$, so
$\mathcal P_6(28)$ contains every realized real type. The sets
$\mathcal P_6(D)$ are nested. Equality of the counts at $24$ and $28$
and the smaller count at $23$ give another proof of $H_6(2)=24$.
Now $H_7(2)\le 2^5+24=56$. All $36{,}519$ types occur by diameter $48$,
and $44$ are absent at $47$. This proves $H_7(2)=48$ and completes
Theorem~\ref{M-thmC}.

\begin{remark}\label{M-rem:six-attribution}
The low-rank counts $1,35,306$ evaluate the quantities $\tau_{6,r}(2)$ in
Shin's stratification \cite[Corollary~8.2 and Theorem~8.4]{M-Shin26}.
His Corollary~8.8 gives $80$ as a common bound for weak and ordered
six-letter types at order two, and his Corollary~7.3 gives the sharp
rank-four bound $16$. The maximum over all weak six-letter types is $24$.
\end{remark}

\section{Exact families}\label{M-sec:cluster}

\subsection{Cluster families in any alphabet size}
Recall that $L_k(q)$ is the least diameter of a relation-free $k$-letter
integer alphabet, $k\ge3$; such alphabets exist by
Lemma~\ref{M-lem:append-qc-plus-one}.

\begin{theorem}\label{M-thm:cluster-k}
Let $k\ge3$ and $q\ge2$. Choose $0<\theta_1<\cdots<\theta_{k-2}<1$ with $1$
and all $\theta_i$ rationally independent, and consider the $(k+1)$-letter
rank-one alphabet
\[
 A=\{0,\ q-1,\ q-1+\theta_1,\ \ldots,\ q-1+\theta_{k-2},\ q\}.
\]
Its exact minimum model diameter is
\begin{equation}\label{M-eq:cluster-k}
 M_q(A)=qL_k(q).
\end{equation}
\end{theorem}
\begin{proof}
Rational independence
leaves $(1,-q,0,\ldots,0,q-1)$ as the sole primitive short relation, up to
sign. In an increasing integer model beginning at zero it fixes the second
and last integer images to $(q-1)u,qu$ with $u>0$. The last $k$ letters,
translated by $-(q-1)u$, must be a relation-free integer $k$-set of diameter
$u$, because their real counterpart $\{0,\theta_1,\ldots,\theta_{k-2},1\}$
has no balanced rational relation; so $u\ge L_k(q)$.

Conversely, from a relation-free set $\{0,a_1,\ldots,a_{k-2},c\}$, form
\[
 \{0,\ (q-1)c,\ (q-1)c+a_1,\ \ldots,\ (q-1)c+a_{k-2},\ qc\}.
\]
A relation not using zero would contradict the relation-free translated
core. In a reduced relation of mass $p\le q$ with zero appearing $z\ge1$
times on one side, the remaining terms lie between $(q-1)c$ and $qc$.
Equality requires $(p-z)q\ge p(q-1)$, so $p\ge zq$. Therefore $p=q$, $z=1$
and all terms are forced to the two endpoints. This is exactly the
prescribed relation. Taking $c=L_k(q)$ proves \eqref{M-eq:cluster-k}. The constructed model
is primitive, since its coordinate gcd is the gcd of the minimum core model.
\end{proof}

\subsection{The five-letter case}
\begin{corollary}\label{M-thmG}\label{M-thm:cluster-model}
Choose $0<\theta<\eta<1$ such that $1,\theta,\eta$ are rationally
independent. The minimum diameter of an order-preserving $q$-model of
\[
 \{0,q-1,q-1+\theta,q-1+\eta,q\}
\]
is exactly $qL_4(q)$, for every $q\ge2$.
\end{corollary}
\begin{proof}
Apply Theorem~\ref{M-thm:cluster-k} with $k=4$.
\end{proof}

Distinct unordered $q$-term sums of a relation-free four-set give
$\binom{q+3}{3}$ distinct integers in $[0,qL_4(q)]$. Thus
\[
 L_4(q)\ge\left\lceil\frac{\binom{q+3}{3}-1}{q}\right\rceil
 =\left\lceil\frac{q^2+6q+11}{6}\right\rceil.
\]
Together with the explicit quadratic four-set, this shows that the exact
family above has model diameter of order $q^3$. The exact value of
$L_4(q)$ is determined by Shin \cite[Theorem~10.3]{M-Shin26}:
$L_4(q)=(q+1)(q+2)/2+[q\text{ odd}]$ (Remark~\ref{M-rem:four-ranks}). At $q=2$
and $q=3$ the relation-free four-sets $\{0,1,4,6\}$ and $\{0,1,7,11\}$ give
the models $\{0,6,7,10,12\}$ and $\{0,22,23,29,33\}$ of diameters $12$
and $33$.

\subsection{Appending \texorpdfstring{$q$}{q} times the maximum}
For a finite ordered real alphabet $T$ beginning at zero, with largest letter
$C>0$, recall that $M_q(T)$ denotes its minimum ordered integer $q$-model
diameter.

\begin{lemma}\label{M-lem:append-qmax}
Let $q\ge2$ and let $T$ be as above. Then
\begin{equation}\label{M-eq:append-qmax}
 M_q(T\cup\{qC\})=q\,M_q(T).
\end{equation}
The only new primitive $q$-relation of $T\cup\{qC\}$ is
$(q-1)\cdot0-q\cdot C+qC=0$; the short-relation rank increases by exactly
one, and primitivity of a minimum core model is preserved. In particular
$H_{m+1}(q)\ge qH_m(q)$.
\end{lemma}
\begin{proof}
To identify every new short relation, cancel common terms and orient the
coefficient of $qC$ positively. Its positive numerical sum is at least $qC$.
Its negative sum uses at most $q$ letters of $T$, each at most $C$, and is
therefore at most $qC$. Equality forces the new coefficient to be one, every
negative term to be $C$, and every other positive term to be zero. The only
new primitive $q$-relation is thus
\[
 (q-1)\cdot0-q\cdot C+qC=0 .
\]
All other $q$-relations are precisely those on the original core. In an
increasing integer model translated to start at zero, this new relation
forces its last coordinate to be $q$ times the old last coordinate.
Restricting to the old labels gives a $q$-model of $T$
(Lemma~\ref{M-lem:elementary}), proving the lower bound in
\eqref{M-eq:append-qmax}. Conversely append $q$ times the maximum to a minimum
model of $T$; the same extremal-sum argument proves that no unwanted
$q$-relation appears, giving equality.

The short-relation rank increases by exactly one, since the new row uses the
new coordinate and every old row has zero there. Primitivity of a minimum
core model is preserved. The final inequality follows by applying
\eqref{M-eq:append-qmax} to an $m$-letter alphabet with $M_q(T)=H_m(q)$.
\end{proof}

The monotone lower estimate $H_{m+1}(q)\ge qH_m(q)$ does not determine the
higher-letter upper bound.

\subsection{Consequences for six letters at order two}
\begin{corollary}\label{M-cor:six-explanations}
\leavevmode
\begin{enumerate}
\item $L_5(2)=11$, attained by $\{0,1,4,9,11\}$, and the rank-one
 six-letter family of Theorem~\ref{M-thm:cluster-k} with $k=5$ and $q=2$
 has minimum diameter $22$, attained by the model
 $\{0,11,12,15,20,22\}$; this is the rank-one maximum of
 Proposition~\ref{M-prop:six-classification}.
\item Appending $q^2$ to the five-letter cluster alphabet of
 Corollary~\ref{M-thm:cluster-model} gives a six-letter alphabet of rank two
 with exact minimum diameter $q^2L_4(q)$. At $q=2$ this is $24$, with model
 $\{0,6,7,10,12,24\}$, one of the rank-two maxima of
 Proposition~\ref{M-prop:six-classification}. At $q=3$ it is $99$, below the
 universal six-letter lower bound $108$ of Theorem~\ref{M-thm:general-bounds}.
\item Repeating the power extension after the five-letter cluster gives, for
 every $m\ge5$, an $m$-letter family of short-relation rank $m-4$ and exact
 minimum diameter $q^{m-4}L_4(q)$; at $q=2$ this is $3\cdot2^{m-3}$.
\end{enumerate}
\end{corollary}
\begin{proof}
(1) The alphabet $\{0,1,4,9,11\}$ has fifteen distinct unordered two-term
sums, so $L_5(2)\le11$. The family of Theorem~\ref{M-thm:cluster-k} with
$k=5$ and $q=2$ has the profile of the single hyperplane
$t_0-2t_1+t_5=0$, and Proposition~\ref{M-prop:six-classification} gives its
exact minimum diameter $22$; by \eqref{M-eq:cluster-k}, $2L_5(2)=22$. The
model displayed is the minimum model recorded there, and it is the
construction of Theorem~\ref{M-thm:cluster-k} applied to $\{0,1,4,9,11\}$.
(2) By Lemma~\ref{M-lem:append-qmax} and Corollary~\ref{M-thm:cluster-model} the
rank is two and the minimum diameter is $q\cdot qL_4(q)$; at $q=2$,
$L_4(2)=6$ and the construction applied to $\{0,1,4,6\}$ gives
$\{0,6,7,10,12\}\cup\{24\}$. At $q=3$, $L_4(3)=11$.
(3) Iterate Lemma~\ref{M-lem:append-qmax}; each step raises the rank by one
and multiplies the minimum diameter by $q$, and $L_4(2)=6$.
\end{proof}

At order two, item~(3) attains the same diameter as the family of
Theorem~\ref{M-thm:corank-model} in a different relation rank.

\section{Further rank-one examples}\label{M-sec:rank-sensitive}

The mass bound of Proposition~\ref{M-prop:rank-one-mass} controls the
small-mass profiles. The examples below compare larger masses with the
cluster value $qL_4(q)$.

\begin{remark}[the range covered by $qL_4(q)$]\label{M-rem:mass-range}
By Shin's formula $L_4(q)=\binom{q+2}{2}+[q\text{ odd}]$
\cite[Theorem~10.3]{M-Shin26} (Remark~\ref{M-rem:four-ranks}),
Proposition~\ref{M-prop:rank-one-mass} shows that every five-letter rank-one
profile of primitive mass $\mu\le\lfloor q/2\rfloor+1$ has a model of
diameter at most $qL_4(q)$, the exact minimum of the cluster family of
Corollary~\ref{M-thm:cluster-model}. For even $q$ this includes the boundary
mass $\mu=q/2+1$, where $\mu q(q+1)=qL_4(q)$ exactly; for odd $q$ the range
is $\mu\le(q+1)/2$. The maximum over the remaining larger masses is
the subject of Question~\ref{M-q:rank-one-mass}.
\end{remark}

\subsection{A proper-power family of mass \texorpdfstring{$q-1$}{q-1}}

\begin{proposition}\label{M-prop:proper-power}
Let $q=2r\ge4$ with $q\equiv0$ or $4\pmod6$, that is, $r\not\equiv1\pmod3$.
The alphabet
\begin{equation}\label{M-eq:proper-power-alphabet}
 \{0,\ r-1,\ r(q+1),\ r(q+4),\ (q-1)r(q+4)\}
\end{equation}
is a primitive increasing integer $q$-model of $\{0,\theta,\eta,1,q-1\}$,
where $0<\theta<\eta<1$ and $1,\theta,\eta$ are rationally independent: up
to sign, its only $q$-relation is
\[
 \delta=(q-2,\,0,\,0,\,-(q-1),\,1),
\]
of mass $\mu=q-1$. Its diameter is $(q-1)r(q+4)=qL_4(q)-3q$.
\end{proposition}

On $\{0,\theta,\eta,1,q-1\}$ rational independence forces the coefficients
of $\theta$ and $\eta$ in a vanishing relation to be zero, leaving the
integer multiples of $\delta$; since $2(q-1)>q$, only $\pm\delta$ have mass
at most $q$. The claim is therefore that \eqref{M-eq:proper-power-alphabet}
satisfies $\delta$ and no other $q$-relation.

\begin{proof}
Set $a=r-1$, $B=q+1$, $C=q+4$ and $D=(q-1)C$, so that
\eqref{M-eq:proper-power-alphabet} is $\{0,a,rB,rC,rD\}$.

\smallskip\noindent\emph{The four scaled anchors.}
Consider first $S=\{0,B,C,D\}$. Since $q\equiv0$ or $1\pmod3$,
$\gcd(B,C)=1$. The primitive balanced relation on $\{0,B,C\}$ is
$(B-C,C,-B)=(-3,q+4,-(q+1))$, of mass $C=q+4>q$, so the three-point core
has no $q$-relation.

A $q$-relation involving $D$ can be oriented so that its coefficient at
$D$ is positive. Since $2D=2(q-1)C>qC$ and the negative side uses at most
$q$ letters, each at most $C$, that coefficient is one. Its mass $p$ is at
least $q-1$, because the negative side must reach $D=(q-1)C$. If $p=q-1$,
equality forces every negative letter to be $C$ and every other positive
letter to be zero, which after cancellation and zero padding is exactly
$D=(q-1)C$.

If $p=q$, let $P$ be the multiset of the $q-1$ other positive old terms,
zeros included, and $N$ the multiset of the $q$ negative old terms. The
equality $D+\sum P=\sum N$ can be rewritten, using
$\sum N=qC-\sum_{x\in N}(C-x)$, as
\begin{equation}\label{M-eq:deficit-identity}
 C=\sum_{x\in N}(C-x)+\sum_{y\in P}y .
\end{equation}
The positive old values are $0$, $B$ or $C$, and $2B>C$, so $P$ has at
most one nonzero value. The nonzero deficits $C-x$ are $3$ (from $x=B$)
and $C$ (from $x=0$). If $P$ contains $C$, every deficit is zero, $N$
consists of $q$ copies of $C$, and canceling one $C$ leaves $D=(q-1)C$.
If $P$ contains $B$, the deficits total $3$, so $N$ contains exactly one
$B$ and $q-1$ copies of $C$, and canceling $B$ leaves $D=(q-1)C$. If $P$
has no nonzero term, then either a single deficit $C$ accounts for the sum,
so that $N$ is one zero and $q-1$ copies of $C$ and canceling the zero
leaves $D=(q-1)C$, or the sum consists of deficits $3$ alone and is a
multiple of $3$, which would require $3\mid C=q+4$, excluded by the
residue hypothesis. Hence every $q$-relation on $S$ is a multiple of the
row $(q-2,0,-(q-1),1)$ on $(0,B,C,D)$, and only its two primitive signs
have mass at most $q$.

\smallskip\noindent\emph{Excluding the inserted free coefficient.}
Since $\gcd(a,r)=1$, reduction modulo $r$ shows that the coefficient of
$a$ in a $q$-relation on $\{0,a,rB,rC,rD\}$ is a multiple of $r$, hence
one of $0,\pm r,\pm2r$. If it is zero, dividing the numerical equation by
$r$ gives a relation on $S$, and the anchor result leaves only the
required line. If its absolute value is $2r=q$, orient it positively; it
consumes the entire positive mass, and dividing the numerical equality by
$r$ would express $2a=q-2$ as a sum of $q$ members of $\{0,B,C,D\}$, which
is impossible because every nonzero member is at least $B=q+1$.

It remains to exclude the coefficient $r$, after reversing the row if
necessary. Let $y,z,t$ be the coefficients at $rB,rC,rD$. The positive
mass already includes $r$, so $y\le r$, while $y\ge-q=-2r$. Dividing the
numerical equation by $r$ gives $a+By+Cz+Dt=0$. Since $B=C-3$ and
$D=(q-1)C$, put $v=y+z+(q-1)t$ to obtain
\begin{equation}\label{M-eq:proper-power-residue}
 3y-Cv=a=r-1 .
\end{equation}
The bounds $-2r\le y\le r$ and $C=2r+4$ give $-4<v<1$, so
$v\in\{-3,-2,-1,0\}$. For $v=0,-1,-3$ the value
$y=(r-1+(2r+4)v)/3$ is an integer only when $r\equiv1\pmod3$, which is
excluded. Thus $v=-2$, and \eqref{M-eq:proper-power-residue} gives
\begin{equation}\label{M-eq:proper-power-v}
 y=-r-3,\qquad z+(2r-1)t=r+1 .
\end{equation}
If $t\le0$, then $z\ge r+1$, exceeding the remaining positive-mass budget
$r$. If $t\ge1$, then $z\le2-r\le0$, and the negative coefficients $y,z$
alone have total mass
\[
 (r+3)+\bigl((2r-1)t-r-1\bigr)=(2r-1)t+2\ge2r+1>q ,
\]
which is also impossible. Hence the free coefficient is zero, and
\eqref{M-eq:proper-power-alphabet} has exactly the profile of $\delta$.

\smallskip\noindent\emph{Order, primitivity and diameter.}
For $r\ge2$ the five letters $0<r-1<rB<rC<rD$ are strictly increasing.
Their gcd is one because $\gcd(B,C)=1$ and $\gcd(a,r)=1$. Finally
\[
 (q-1)r(q+4)=\frac{q(q-1)(q+4)}2=\frac{q(q+1)(q+2)}2-3q=qL_4(q)-3q,
\]
by the even-$q$ formula $L_4(q)=(q+1)(q+2)/2$ of
\cite[Theorem~10.3]{M-Shin26}.
\end{proof}

\begin{remark}[The congruence condition]\label{M-rem:proper-power-scope}
When $q\equiv2\pmod6$ both $B=q+1$ and $C=q+4$ are divisible by $3$, and
the two-anchor relation $(B-C,C,-B)/3$ on $\{0,B,C\}$ has mass
$(q+4)/3\le q$ for every choice of the inserted point. Thus these anchors
introduce an extra short relation in this residue class. At $q=4$ and $q=6$ the alphabet
\eqref{M-eq:proper-power-alphabet} is $\{0,1,10,16,48\}$ and
$\{0,2,21,30,150\}$. For $q\ge6$ the mass $\mu=q-1$ exceeds
$\lfloor q/2\rfloor+1$, so this family lies outside the range of
Remark~\ref{M-rem:mass-range} while still having a model below $qL_4(q)$.
This construction gives an upper bound for the minimum diameter of the profile.
\end{remark}

\section{Finite additive-square-free spectra}\label{M-sec:application}

A word over a real alphabet $T$ is a finite sequence of its elements.
An \emph{additive square} consists of two adjacent nonempty blocks of
equal length and equal sum. A word containing no such factor is
\emph{additive-square-free}. Write
\[
 g(T)=\sup\{|W|:W\text{ is an additive-square-free word over }T\},
\]
with $g(T)=\infty$ if these lengths are unbounded.
Freedman and Brown \cite{M-FB16} began the study of the finite values of $g$
on four-element real alphabets, and the companion paper \cite{M-Spec26}
classifies the four-element real alphabets with $g(T)<139$ together with
the real line $\{0,1,x,x+2\}$ below $167$. The \emph{finite spectrum} on
$m$ real (respectively integer) letters is the set of finite values of
$g(T)$ over $m$-element real (integer) alphabets. A \emph{labelled word} is
a word over the labels $\{0,\ldots,m-1\}$; it is read on an ordered
$m$-alphabet through $i\mapsto t_i$.

\begin{lemma}\label{M-lem:good-transfer}
If two ordered $m$-alphabets $T$ and $T'$ are $q$-equivalent, then a
labelled word of length at most $2q+1$ contains an additive square on $T$ if and only if it contains one on $T'$.
\end{lemma}
\begin{proof}
A factor $UV$ with $|U|=|V|=l$ of a word of length at most $2q+1$ has
$l\le q$. The equality $\sum U=\sum V$ is the vanishing on the alphabet of
the balanced relation $\pi(U)-\pi(V)$, where $\pi$ records the multiplicity
of each label, and this relation has mass at most $l\le q$. By
Definition~\ref{M-def:model} it vanishes on $T$ if and only if it vanishes on
$T'$.
\end{proof}

\Needspace{14\baselineskip}
\begin{proposition}\label{M-prop:spectrum-transfer}
\leavevmode
\begin{enumerate}
\item The finite spectrum on four real letters equals the finite spectrum
 on four integer letters. For each integer cutoff $L\ge4$, every finite
 value $g(T)<L$ of a four-element real alphabet is realized by a primitive
 integer alphabet of diameter at most
 $\lfloor L/2\rfloor(\lfloor L/2\rfloor+1)$.
\item For every $m\ge2$ the finite spectrum on $m$ real letters equals the
 finite spectrum on $m$ integer letters, and for each integer $L\ge4$ every
 finite value $g(T)<L$ of an $m$-element real alphabet is realized by a
 primitive integer alphabet of diameter at most $H_m(\lfloor L/2\rfloor)$.
 For $m=5$ this bound is $\lfloor L/2\rfloor^2(\lfloor L/2\rfloor+1)$, and
 for $m\ge5$ it is at most $\sum_{j=2}^{m-2}\lfloor L/2\rfloor^j$.
\end{enumerate}
\end{proposition}
\begin{proof}
Integer alphabets are real alphabets, so one inclusion is trivial. Let $T$
be an $m$-element real alphabet with $g(T)=N<\infty$ and choose
$q=\max(2,\lfloor(N+1)/2\rfloor)$, so that $2q+1\ge N+1$. Let $T'$ be a
primitive order-preserving $q$-model of $T$, which exists by
Theorem~\ref{M-thm:flag-model}. By Lemma~\ref{M-lem:good-transfer}, $T'$ has a
word of length $N$ avoiding additive squares and no such word of length
$N+1$; any longer avoiding word would have an avoiding prefix of length $N+1$. Hence $g(T')=N$,
which proves the equality of the finite spectra for every $m$.

If $N<L$, take instead $q=\lfloor L/2\rfloor\ge2$; then $N+1\le L\le2q+1$
and the same argument gives $g(T')=N$ for a primitive $q$-model $T'$ of
diameter at most $H_m(q)$. For $m=4$ this is $q(q+1)$ by
Theorem~\ref{M-thm:sharp-four-model}, for $m=5$ it is $q^2(q+1)$ by
Theorem~\ref{M-thm:sharp-five-model}, and for $m\ge5$ it is at most
$\sum_{j=2}^{m-2}q^j$ by Theorem~\ref{M-thm:general-bounds}.
\end{proof}

\begin{remark}[Finite real and integer spectra]\label{M-rem:transfer-classical}
Equality of the finite spectra on real and integer
letters follows from the classical finite Freiman-model principle
\cite[Part~II]{M-GR09}, and it is contained in the
real-to-integer realizability statements of Nathanson \cite{M-Nat18} and
O'Bryant \cite{M-OBr25}. It also follows from rational approximation alone:
the finitely many required zero relations define a rational subspace,
strict letter order and avoidance of the other short-relation hyperplanes
are open conditions there, rational points are dense in that subspace, and
clearing denominators produces an integer model. Proposition~\ref{M-prop:spectrum-transfer}
bounds the diameter of such a model in terms of the observation length.
The model may depend on that length. For infinite words, Shallit has
constructed an additive-square-free example on six positive integer letters
\cite{M-Shallit26}.
\end{remark}

\begin{remark}[a fixed relation line]\label{M-rem:relation-line}
There is also an anisotropic consequence on a fixed relation line. For four
letters, one primitive relation of mass $\mu$ together with an independent
relation of mass at most $q$ forces the alphabet to be affinely equivalent
to a primitive integer alphabet of diameter at most $\mu q$, by
Lemma~\ref{M-lem:mass-minor}: the common kernel is a rational line whose
primitive generator has height at most $\mu q$. Fix the line of ordered
four-letter alphabets on which one primitive relation of mass $\mu$
vanishes, and a labelled word $W$ of length $L$. Each block comparison of
$W$ restricts to an affine function of the line parameter. If $W$ avoids additive squares
at one point of the line, no comparison vanishes identically on the line;
a point where $W$ fails therefore carries a relation of mass at most
$\lfloor L/2\rfloor$ independent of the fixed one, so it is a rational
direction of primitive diameter at most $\mu\lfloor L/2\rfloor$, and there
are only finitely many such directions. On the line $\{0,1,x,x+2\}$,
$x>1$, the fixed relation $(2,-2,-1,1)$ has $\mu=3$, and a word of length
$L=167$ gives the bound $249$. The companion paper \cite{M-Spec26} exhibits
four labelled words of length $167$ on this line and lists their exact
rational failure sets.
\end{remark}

\section{Open problems}\label{M-sec:open}

\begin{question}[the universal radius]\label{M-q:six}\label{M-q:stratum}
Is
\[
 H_m(q)=q^{m-3}(q+1)
\]
for every $m\ge6$ and $q\ge2$?
\end{question}
The value on the right is attained in relation rank $m-3$ by
Theorem~\ref{M-thm:corank-model}. It is the universal maximum for
$m=3,4,5$ at every order $q\ge2$, and for $m=6,7$ at $q=2$.
For six letters and $q\ge3$ the bounds leave
$q^4+q^3\le H_6(q)\le q^4+q^3+q^2$.

The question concerns the maximum over all relation ranks. For $m=4$,
the per-rank maxima are $L_4(q)$, $q(q+1)$ and $q^2$
(Remark~\ref{M-rem:four-ranks}). For $m=5$, rank two attains $q^2(q+1)$
and the same number bounds every rank. At $(m,q)=(6,2)$, the maxima in
ranks $0,1,2$ are $17,22,24$, respectively
(Proposition~\ref{M-prop:six-classification}); rank three also attains
$24$, while rank four is bounded by $16$. Thus the universal maximum
need not be confined to rank $m-3$.

\begin{question}[relation-free five-sets]\label{M-q:L5}
Determine $L_5(q)$, the least diameter of a five-element integer alphabet
with no nonzero $q$-relation. Rank zero is a single profile, so $L_5(q)$
is also the maximum of $M_q$ over relation-free five-letter alphabets.
Corollary~\ref{M-cor:five-rank-zero} gives $L_5(q)\le q^2(q+1)$.
At $q=2$ the optimal five-mark Golomb ruler has length $11$
\cite[Table~2.2]{M-Dim02}. At $q=3,4$ the minima are $23$ and $41$,
respectively (Section~\ref{M-sec:five}). A formula at all orders remains open.
\end{question}

\begin{question}[rank-one five-letter profiles by mass]\label{M-q:rank-one-mass}
For $2\le\mu\le q$, determine the maximum of $M_q(T)$ over five-letter
rank-one profiles of primitive mass $\mu$, and in particular whether the
maximum over all rank-one five-letter profiles is $qL_4(q)$, the value of
the cluster family of Corollary~\ref{M-thm:cluster-model}.
Proposition~\ref{M-prop:rank-one-mass} gives the upper bound $\mu q(q+1)$,
which is at most $qL_4(q)$ exactly when $\mu\le\lfloor q/2\rfloor+1$
(Remark~\ref{M-rem:mass-range}); the cluster family has mass $q$ and
diameter $qL_4(q)$; and the proper-power family of
Proposition~\ref{M-prop:proper-power} shows that profiles of mass
$q-1>\lfloor q/2\rfloor+1$, for $q\ge6$ with $q\equiv0,4\pmod6$, also
admit models below $qL_4(q)$, namely of diameter $qL_4(q)-3q$. The
remaining high-mass profiles are covered by neither statement.
\end{question}

\begin{remark}[label-level realization lengths]\label{M-rem:label-length}
Following Nathanson \cite{M-Nat26}, let $N(h,k)$ be the least $N$ such that
every value of $|hA|$ over $k$-element sets $A\subset\Z$ is attained by a
$k$-element set contained in $[0,N-1]\cap\Z$.
Nathanson \cite[Theorem~8]{M-Nat26} gives $N(h,k)<4(8h)^{k-1}$ for
$h,k\ge3$. His proof, using Lemma~1, also yields an integer Freiman
$h$-isomorphic copy within that diameter bound: minimize the centered
radius in the isomorphism class and apply the same compression step.
That correspondence need not preserve the order of the elements.
 Since $|hA|$ depends only on
the $h$-zero profile of $A$ under the increasing correspondence,
Corollary~\ref{M-cor:label} gives
\[
 N(h,k)\le1+H_k(h)\le1+\sum_{j=0}^{k-2}h^j\qquad(h\ge2,\ k\ge3),
\]
and $N(h,k)\le1+\sum_{j=2}^{k-2}h^j$ for $k\ge5$. Shin
\cite[Theorem~1.5(ii), Corollary~7.4]{M-Shin26} proves
$N(h,k)\le1+\max\{(k-2)h^{k-2},\sum_{j=0}^{k-2}h^j\}$; for $k\ge4$ and
$h\ge2$ the maximum is $(k-2)h^{k-2}$, so the bound above is smaller. At
$k=4$ his Corollary~10.5, $\nu(h,4)\le\max\{h^2,D_h^{\max}\}$ with
$\nu(h,4)=N(h,4)-1$ and $D_h^{\max}=L_4(h)$, is at least as good as
$H_4(h)=h(h+1)$, and better for $h\ge3$.
\end{remark}

\begin{remark}[Cardinalities for four elements]\label{M-q:shin-label}
Shin asked after Corollary~10.5 of \cite{M-Shin26} whether the four-mark
$B_h$-ruler diameter $L_4(h)$ realizes every attainable cardinality.
The exact label diameter is $N(h,4)-1=L_4(h)$ for $2\le h\le10$
and $(h-3)(h-1)$ for $h\ge11$ \cite{M-Sumset26}. For example, at
$h=11$ the largest cardinality fits in diameter $79$, whereas realizing
all cardinalities requires diameter $80$. Preserving every short
equality in the increasing correspondence requires $H_4(11)=132$.
\end{remark}

\subsection*{Data and code availability}
Two programs accompany this version as arXiv ancillary files.
The file \texttt{models\_q2.cpp} enumerates all six- and seven-letter
equality patterns in Section~\ref{M-sec:all-types};
\texttt{six\_letter\_ranks.py} reproduces the six-letter table by rank.
The earlier low-rank certificates and small-order examples are in the
September archive \url{https://doi.org/10.5281/zenodo.22678365}.

\subsection*{Use of generative AI}
Generative AI tools, including OpenAI Codex, assisted with mathematical
analysis, research programming, literature review and manuscript preparation.
The author takes responsibility for its contents.

\clearpage
\hypertarget{supplement-start}{}
\setcounter{section}{0}\setcounter{subsection}{0}
\setcounter{equation}{0}
\setcounter{table}{0}\setcounter{figure}{0}
\renewcommand{\MainPrefix}{Main }\renewcommand{\SuppPrefix}{}
\renewcommand{\thesection}{S\arabic{section}}
\renewcommand{\thetable}{S\arabic{table}}
\section*{Technical supplement}

The six-letter classification below separates the equality types by rank.
Section~\ref{M-sec:all-types} counts all types at once for six and seven
letters. The tables in this supplement give the low-rank six-letter
models and some examples at higher orders. The final section gives
another proof of the bound for two-dimensional kernels.

Throughout, $q$ is the
observation order, an ordered alphabet is $T=(t_0<\cdots<t_{m-1})$, a
balanced relation $\delta\in\Z^m$ has $\sum_i\delta_i=0$ and mass
$\mass(\delta)=\sum_i\max(\delta_i,0)$, the $q$-zero profile is the set of
balanced relations of mass at most $q$ vanishing on $T$, $H_m(q)$ is the
least universal diameter of an order-preserving $q$-model of an ordered real
$m$-alphabet, and $L_k(q)$ is the least diameter of a $k$-element integer
alphabet with no nonzero balanced relation of mass at most $q$.

\section{The six-letter classification at order two}\label{S-supp:six}

For six letters, Theorem~C gives $H_6(2)=24$. The upper bound for complete short-relation
rank at most two rests on the finite classification of
Proposition~\ref{M-prop:six-classification} of the main article, whose data are recorded here; ranks
three and four use Theorem~\ref{M-thm:corank-model} and Corollary~\ref{M-cor:max-rank}, as recalled in
\S\ref{S-supp:high-ranks}.

\subsection{The thirty-five hyperplanes}\label{S-supp:hyperplanes}
For six letters and $q=2$, every nonzero balanced relation of mass at most
two that can vanish on a strictly increasing alphabet is, up to sign, one
of the $\binom63=20$ relations $t_i+t_k=2t_j$ with $i<j<k$ or one of the
$\binom64=15$ relations $t_i+t_l=t_j+t_k$ with $i<j<k<l$ (main article,
Section~\ref{M-sec:six}). Table~\ref{S-tab:hyperplanes} lists them in the order used by the
data files: rows $1$--$20$ are the triples $(i,j,k)$ in lexicographic order
and rows $21$--$35$ are the quadruples $(i,j,k,l)$ in lexicographic order.
Row $n$ corresponds to bit $n-1$ of the zero masks described in
\S\ref{S-supp:flats}. For each row the table gives the balanced vector
$\delta=(\delta_0,\ldots,\delta_5)$ with $\delta\cdot T=0$; its gap row is
$(A_1,\ldots,A_5)$ with $A_j=\sum_{i\ge j}\delta_i$, so that
$\delta\cdot T=\sum_jA_jg_j$ on the gaps $g_j=t_j-t_{j-1}$.

\begin{table}[htbp]
\caption{The $35$ order-two hyperplanes on six letters, in data-file order.}
\label{S-tab:hyperplanes}
\begin{center}\small
\begin{tabular}{@{}rll@{\qquad}rll@{}}
\toprule
no. & equality & $\delta$ & no. & equality & $\delta$\\
\midrule
1 & $t_0+t_2=2t_1$ & $(1,-2,1,0,0,0)$ & 19 & $t_2+t_5=2t_4$ & $(0,0,1,0,-2,1)$ \\
2 & $t_0+t_3=2t_1$ & $(1,-2,0,1,0,0)$ & 20 & $t_3+t_5=2t_4$ & $(0,0,0,1,-2,1)$ \\
3 & $t_0+t_4=2t_1$ & $(1,-2,0,0,1,0)$ & 21 & $t_0+t_3=t_1+t_2$ & $(1,-1,-1,1,0,0)$ \\
4 & $t_0+t_5=2t_1$ & $(1,-2,0,0,0,1)$ & 22 & $t_0+t_4=t_1+t_2$ & $(1,-1,-1,0,1,0)$ \\
5 & $t_0+t_3=2t_2$ & $(1,0,-2,1,0,0)$ & 23 & $t_0+t_5=t_1+t_2$ & $(1,-1,-1,0,0,1)$ \\
6 & $t_0+t_4=2t_2$ & $(1,0,-2,0,1,0)$ & 24 & $t_0+t_4=t_1+t_3$ & $(1,-1,0,-1,1,0)$ \\
7 & $t_0+t_5=2t_2$ & $(1,0,-2,0,0,1)$ & 25 & $t_0+t_5=t_1+t_3$ & $(1,-1,0,-1,0,1)$ \\
8 & $t_0+t_4=2t_3$ & $(1,0,0,-2,1,0)$ & 26 & $t_0+t_5=t_1+t_4$ & $(1,-1,0,0,-1,1)$ \\
9 & $t_0+t_5=2t_3$ & $(1,0,0,-2,0,1)$ & 27 & $t_0+t_4=t_2+t_3$ & $(1,0,-1,-1,1,0)$ \\
10 & $t_0+t_5=2t_4$ & $(1,0,0,0,-2,1)$ & 28 & $t_0+t_5=t_2+t_3$ & $(1,0,-1,-1,0,1)$ \\
11 & $t_1+t_3=2t_2$ & $(0,1,-2,1,0,0)$ & 29 & $t_0+t_5=t_2+t_4$ & $(1,0,-1,0,-1,1)$ \\
12 & $t_1+t_4=2t_2$ & $(0,1,-2,0,1,0)$ & 30 & $t_0+t_5=t_3+t_4$ & $(1,0,0,-1,-1,1)$ \\
13 & $t_1+t_5=2t_2$ & $(0,1,-2,0,0,1)$ & 31 & $t_1+t_4=t_2+t_3$ & $(0,1,-1,-1,1,0)$ \\
14 & $t_1+t_4=2t_3$ & $(0,1,0,-2,1,0)$ & 32 & $t_1+t_5=t_2+t_3$ & $(0,1,-1,-1,0,1)$ \\
15 & $t_1+t_5=2t_3$ & $(0,1,0,-2,0,1)$ & 33 & $t_1+t_5=t_2+t_4$ & $(0,1,-1,0,-1,1)$ \\
16 & $t_1+t_5=2t_4$ & $(0,1,0,0,-2,1)$ & 34 & $t_1+t_5=t_3+t_4$ & $(0,1,0,-1,-1,1)$ \\
17 & $t_2+t_4=2t_3$ & $(0,0,1,-2,1,0)$ & 35 & $t_2+t_5=t_3+t_4$ & $(0,0,1,-1,-1,1)$ \\
18 & $t_2+t_5=2t_3$ & $(0,0,1,-2,0,1)$ & & & \\
\bottomrule
\end{tabular}
\end{center}
\end{table}

\subsection{Candidate masks and flats}\label{S-supp:flats}
A real profile of rank $r\le2$ has a basis of $r$ rows among the $35$, and
its complete zero profile is the set of all $35$ rows lying in the rational
span of that basis. There are
\[
 1+35+\binom{35}2=631
\]
subsets of at most two rows. Their row-space closures give exactly $487$
distinct sets of rows, the \emph{flats}: one of rank zero (the empty
profile), $35$ of rank one (each single row is its own closure, since two
distinct primitive rows are not proportional), and $451$ of rank two. A
flat is stored as a \emph{zero mask}, the $35$-bit integer whose bit
$n-1$ is set exactly when row $n$ of Table~\ref{S-tab:hyperplanes} lies in
the flat; the two data files below index the rows of Table~\ref{S-tab:hyperplanes}
from $0$ (\texttt{basis\_indices}) and from $1$
(\texttt{basis\_hyperplane\_ids\_1\_based}) respectively.

A flat is \emph{feasible} if some strictly increasing real alphabet has
exactly that zero profile. Of the $487$ flats, $342$ are feasible, namely
$1$, $35$ and $306$ in ranks $0$, $1$ and $2$; the other $145$, all of rank
two, are infeasible. Two different exact certificates establish this
split.

\emph{The producer} (\texttt{classify\_q2\_low\_rank.py}) works with the
basis gap rows $E$ of a flat and the polytope
$\{g\in\R^5:Eg=0,\ g\ge0,\ \sum_jg_j=1\}$, whose vertices it computes in
exact rational arithmetic. The flat is declared feasible when the vertex
set is nonempty and every coordinate is positive at some vertex, in which
case the average of the vertices is a strictly positive gap vector; a
strictly positive point is a convex combination of vertices, so the
criterion is exact. For an infeasible flat the record notes either that
the polytope is empty or the index of a gap that vanishes at every vertex
(\texttt{forced\_zero\_gap}). For every feasible flat the producer also
finds a rational point of the polytope whose zero mask is exactly the flat,
by combining the vertices with the weights $1,w,w^2,\ldots$ for an integer
$w$ and excluding the finitely many weights at which some further row
vanishes, and clears denominators; the resulting integer alphabet
(\texttt{real\_type\_integer\_witness}) has exactly that zero profile but
is not minimal.

\emph{The verifier} (\texttt{verify\_q2\_low\_rank.py}) uses a different
certificate for infeasibility: for each of the $145$ infeasible flats it
records an integer combination $c_1E_1+c_2E_2$ of the two basis gap rows
that is a nonzero vector with nonnegative coordinates. Such a vector has
positive inner product with every strictly positive gap vector, while the
basis equations would make that product zero; hence no strictly increasing
real alphabet lies on the flat. For example, the flat with zero mask $3$
has basis rows $1$ and $2$ (the equalities $t_0+t_2=2t_1$ and
$t_0+t_3=2t_1$), with gap rows $(-1,1,0,0,0)$ and $(-1,1,1,0,0)$; the
combination $(-1,1)$ gives $(0,0,1,0,0)$, that is, the two equalities force
$g_3=t_3-t_2=0$. The verifier also recomputes all $487$ closures from
scratch: it derives the $35$ rows from the differences of the $21$
multiplicity vectors of unordered two-term sums, allowing repeated letters, keeping exactly the primitive rows whose
gap coefficients have both signs, and computes ranks over the field of
$101$ elements, which is exact here because every minor involved has
absolute value at most $3!\cdot2^3=48<101$.

\subsection{Exact minimum models}\label{S-supp:models}
The integer alphabets beginning at $0$ with diameter at most $24$ are the
$\binom{24}5=42{,}504$ choices of five positive letters in
$\{1,\ldots,24\}$. Of these, $1{,}470$ have complete short-relation rank
$0$, $12{,}580$ rank $1$, $20{,}321$ rank $2$ and $8{,}133$ rank at least
$3$. The producer evaluates the literal two-term-sum collisions of each
alphabet, the verifier evaluates the $35$ equalities of
Table~\ref{S-tab:hyperplanes} directly, and both obtain the same dictionary
assigning to each of the $342$ feasible flats its minimum-diameter model
(ties broken lexicographically). Every feasible flat occurs within the
cutoff, so the minima are exact, and every minimum model is primitive.
Table~\ref{S-tab:distribution} gives the distribution of the minimum
diameters; the maxima are $17$, $22$ and $24$ in ranks $0$, $1$ and $2$.

\begin{table}[htbp]
\caption{Number of feasible flats with a given minimum model diameter.}
\label{S-tab:distribution}
\begin{center}\small
\begin{tabular}{@{}lrrrrrrrrrrrrrr@{}}
\toprule
diameter & 10 & 11 & 12 & 13 & 14 & 15 & 16 & 17 & 18 & 19 & 20 & 22 & 24 & total\\
\midrule
rank 0 & & & & & & & & 1 & & & & & & 1\\
rank 1 & & & & 2 & 4 & 5 & 12 & 6 & & 4 & & 2 & & 35\\
rank 2 & 3 & 10 & 29 & 46 & 28 & 56 & 48 & 28 & 24 & 8 & 14 & 8 & 4 & 306\\
\bottomrule
\end{tabular}
\end{center}
\end{table}

The extremal profiles are the following, with their basis rows numbered as
in Table~\ref{S-tab:hyperplanes}.
\begin{center}\small
\begin{tabular}{@{}rlll@{}}
\toprule
rank & basis rows & zero profile & minimum model\\
\midrule
$0$ & none & empty & $\{0,1,4,10,12,17\}$\\
$1$ & $4$ & $t_0+t_5=2t_1$ & $\{0,11,12,15,20,22\}$\\
$1$ & $10$ & $t_0+t_5=2t_4$ & $\{0,1,4,9,11,22\}$\\
$2$ & $3,10$ & $t_0+t_4=2t_1,\ t_0+t_5=2t_4$ & $\{0,6,7,10,12,24\}$\\
$2$ & $8,10$ & $t_0+t_4=2t_3,\ t_0+t_5=2t_4$ & $\{0,1,4,6,12,24\}$\\
$2$ & $4,13$ & $t_0+t_5=2t_1,\ t_1+t_5=2t_2$ & $\{0,12,18,19,22,24\}$\\
$2$ & $4,16$ & $t_0+t_5=2t_1,\ t_1+t_5=2t_4$ & $\{0,12,13,16,18,24\}$\\
\bottomrule
\end{tabular}
\end{center}
The two rank-one maxima are the cluster family of Theorem~\ref{M-thm:cluster-k} with $k=5$,
$q=2$ (Corollary~\ref{M-cor:six-explanations}), in its two orientations; the first rank-two maximum is
the tower family of Corollary~\ref{M-cor:six-explanations}(2). The rank-one and rank-two profiles
listed here have zero profiles equal to their basis rows, since no further
row lies in the span.

\subsection{The two high ranks}\label{S-supp:high-ranks}
The flat certificates in this supplement cover ranks zero through two.
The census in Section~\ref{M-sec:all-types} gives the counts for all five ranks. Rank three is the case
$r=m-3$ of Theorem~\ref{M-thm:corank-model} with $m=6$ and $q=2$, which gives a primitive model of
diameter at most $(q+1)H\le(2+1)\cdot2^3=24$. Rank four is the
one-dimensional-kernel case of Corollary~\ref{M-cor:max-rank}, which gives diameter at most
$2^4=16$. The lower bound $24$ is the rank-three family
$\{0,\theta,1,2,4,8\}$ of Theorem~\ref{M-thm:corank-model}, whose equality model is
$\{0,1,3,6,12,24\}$.

\subsection{Certificate data}
\label{S-supp:package}
The electronic supplement supplies the $35$ hyperplane rows, all $487$
closed masks, the $342$ minimum integer models and the $145$ nonnegative
dual combinations described above. Each flat is recorded with its mask,
rank and basis, together with its feasibility certificate or minimum model.
The mask and basis-index conventions are those of
\S\ref{S-supp:hyperplanes}. Two separate exact programs reproduce the
classification and the minima among all $42{,}504$ integer candidates;
their methods are described in \S\ref{S-supp:models}. Instructions for
running them accompany the code and data archive.

The finite construction checks cover the four-letter relation-free and
sharp rank-one families, the five-letter relation-free family and the
cluster family for $q=2,\ldots,6$, and the proper-power examples at $q=4,6$.
The radius checks determine the four-letter per-rank maxima for
$q=2,\ldots,5$ and $L_5(q)$ for $q=2,3,4$, and check the displayed
five-letter relation-free construction through $q=8$. These finite checks
support the listed instances; the parameter-uniform results follow from
the proofs in the main article.

\section{Tables}\label{S-supp:tables}

\subsection{The values \texorpdfstring{$L_4(q)$}{L4(q)}}
Shin's Theorem~10.3 \cite{S-Shin26} gives
\[
 L_4(q)=\binom{q+2}2+[q\text{ odd}],
\]
attained by the optimal ruler $\{0,1,1+\tbinom{q+1}2,L_4(q)\}$, which is
the rank-zero alphabet $T_q=\{0,1,b,b+d\}$ of the proof of Theorem~A; after
a shift of the index this is sequence A227589 of the OEIS \cite{S-OEIS}.
Table~\ref{S-tab:L4} lists the values for $q\le8$. The per-rank maxima of
Remark~\ref{M-rem:four-ranks} are also shown: rank one $q(q+1)$ and rank two $q^2$, so that
$H_4(q)=q(q+1)$ is the rank-one value for every $q\ge2$, with the rank-zero
value $L_4(q)$ equal to it only at $q=2$.

\begin{table}[htbp]
\caption{$L_4(q)$, its optimizer, and the four-letter per-rank maxima.}
\label{S-tab:L4}
\begin{center}\small
\begin{tabular}{@{}rrlrr@{}}
\toprule
$q$ & $L_4(q)$ & optimizer $\{0,1,1+\binom{q+1}2,L_4(q)\}$ & rank one $q(q+1)$ & rank two $q^2$\\
\midrule
2 & 6 & $\{0,1,4,6\}$ & 6 & 4\\
3 & 11 & $\{0,1,7,11\}$ & 12 & 9\\
4 & 15 & $\{0,1,11,15\}$ & 20 & 16\\
5 & 22 & $\{0,1,16,22\}$ & 30 & 25\\
6 & 28 & $\{0,1,22,28\}$ & 42 & 36\\
7 & 37 & $\{0,1,29,37\}$ & 56 & 49\\
8 & 45 & $\{0,1,37,45\}$ & 72 & 64\\
\bottomrule
\end{tabular}
\end{center}
\end{table}

\subsection{Witnesses and equality models}
Table~\ref{S-tab:witnesses} collects the explicit alphabets of the main
article: the relation-free (rank-zero) constructions, the extremal real
alphabets that force the sharp values, and the integer models attaining
them. Throughout, $0<\theta<\eta<1$ and $1,\theta,\eta$ are rationally
independent, $b=q(q+1)/2+1$, $d=q$ for even $q$ and $d=q+1$ for odd $q$,
and $c=b+d$.

\begin{table}[htbp]
\caption{Explicit constructions, witnesses and equality models.}
\label{S-tab:witnesses}
\begin{center}\small
\begin{tabular}{@{}lllr@{}}
\toprule
statement & real alphabet or role & integer alphabet & diameter\\
\midrule
Thm.~A, rank $0$ & relation-free model & $\{0,1,b,b+d\}$ & $L_4(q)$\\
Thm.~A, rank $1$ & extremal $\{0,\theta,1,q\}$ & $\{0,1,q+1,q(q+1)\}$ & $q(q+1)$\\
Thm.~A, rank $2$ & sharpness of $q^2$ & $\{0,q,q+1,q^2\}$ & $q^2$\\
Thm.~B, rank $0$, $q\ge3$ & relation-free model & $\{0,1,b,c,qc+1\}$ & $qc+1$\\
Thm.~B, rank $0$, $q=2$ & relation-free model & $\{0,1,4,9,11\}$ & $11$\\
Thm.~B, rank $2$ & extremal $\{0,\theta,1,q,q^2\}$ & $\{0,1,q+1,q(q+1),q^2(q+1)\}$ & $q^2(q+1)$\\
Thm.~\ref{M-thm:corank-model} & corank-two extremal family & its power-family model & $q^{m-3}(q+1)$\\
Thm.~C & extremal $\{0,\theta,1,2,4,8\}$ & $\{0,1,3,6,12,24\}$ & $24$\\
Cor.~\ref{M-thm:cluster-model}, $q=2$ & $\{0,1,1+\theta,1+\eta,2\}$ & $\{0,6,7,10,12\}$ & $12$\\
Cor.~\ref{M-thm:cluster-model}, $q=3$ & $\{0,2,2+\theta,2+\eta,3\}$ & $\{0,22,23,29,33\}$ & $33$\\
Prop.~\ref{M-prop:proper-power}, $q=4$ & $\{0,\theta,\eta,1,3\}$ & $\{0,1,10,16,48\}$ & $48$\\
Prop.~\ref{M-prop:proper-power}, $q=6$ & $\{0,\theta,\eta,1,5\}$ & $\{0,2,21,30,150\}$ & $150$\\
\bottomrule
\end{tabular}
\end{center}
\end{table}

The rank-zero rows are explicit models of the single relation-free profile
and bound $L_4(q)$ and $L_5(q)$ from above; the value $L_4(q)$ is exact by
Shin's theorem. The present paper does not determine $L_5(q)$ for arbitrary $q$; exhaustive search
 gives $L_5(2)=11$ ($\{0,1,4,9,11\}$),
$L_5(3)=23$ ($\{0,1,15,18,23\}$) and $L_5(4)=41$ ($\{0,1,24,37,41\}$),
whereas the construction of Corollary~\ref{M-cor:five-rank-zero} gives $11$, $34$ and $61$, so that
construction is not optimal in the two cases $q=3,4$. Table~\ref{S-tab:instances}
gives the numerical instances of the rank-zero and rank-one rows for
$q\le8$. The finite verification ranges are specified in \S\ref{S-supp:package};
in particular the relation-free constructions are checked through $q=8$.

\begin{table}[htbp]
\caption{Numerical instances of the four- and five-letter constructions.}
\label{S-tab:instances}
\begin{center}\small
\begin{tabular}{@{}rlrlrl@{}}
\toprule
$q$ & $\{0,1,b,b+d\}$ & $L_4(q)$ & $\{0,1,b,c,qc+1\}$ & $qc+1$ & $\{0,1,q+1,q(q+1)\}$\\
\midrule
2 & $\{0,1,4,6\}$ & 6 & $\{0,1,4,9,11\}$ (special) & 11 & $\{0,1,3,6\}$\\
3 & $\{0,1,7,11\}$ & 11 & $\{0,1,7,11,34\}$ & 34 & $\{0,1,4,12\}$\\
4 & $\{0,1,11,15\}$ & 15 & $\{0,1,11,15,61\}$ & 61 & $\{0,1,5,20\}$\\
5 & $\{0,1,16,22\}$ & 22 & $\{0,1,16,22,111\}$ & 111 & $\{0,1,6,30\}$\\
6 & $\{0,1,22,28\}$ & 28 & $\{0,1,22,28,169\}$ & 169 & $\{0,1,7,42\}$\\
7 & $\{0,1,29,37\}$ & 37 & $\{0,1,29,37,260\}$ & 260 & $\{0,1,8,56\}$\\
8 & $\{0,1,37,45\}$ & 45 & $\{0,1,37,45,361\}$ & 361 & $\{0,1,9,72\}$\\
\bottomrule
\end{tabular}
\end{center}
\end{table}

For $q\ge3$ the five-letter rank-zero diameter $qc+1$ is at most
$(q^3+3q^2+4q+2)/2<q^2(q+1)$ (Corollary~\ref{M-cor:five-rank-zero}); at $q=2$ the general
construction would give $\{0,1,4,6,13\}$, and the alphabet $\{0,1,4,9,11\}$
of smaller diameter is used instead.

\section{An alternative lattice proof}\label{S-supp:lattice}

\begin{Sremark}[a lattice proof of the upper bound]\label{S-rem:lattice-proof}
The upper bound $(q+1)H$ can also be obtained by a two-dimensional lattice
subdivision, which we record because it uses only the two extreme rays of
the ordered kernel cone. Let $M$ be the matrix of the $r=m-3$ selected rows,
let $\Lambda=\ker M\cap\Z^{m-1}$, a saturated lattice of rank two, and let
$K=\ker_\R M\cap\{0\le x_1\le\cdots\le x_{m-1}\}$, a two-dimensional pointed
cone with strictly ordered relative interior on which height is positive.
Each extreme ray of $K$ is obtained by imposing at least one ordering
boundary; appending the corresponding unit-transfer row to $M$ produces a
rank-$(m-2)$ matrix with a one-dimensional kernel, whose cofactor vector is
an integer ray generator of height at most $H$ by Lemma~\ref{M-lem:mass-minor},
and making it primitive only reduces the height. If the primitive ray
generators $u_0,v_0$ are a basis of $\Lambda$, stop. Otherwise their index
$D>1$ in $\Lambda$ is witnessed by a lattice point
$w=\alpha u_0+\beta v_0\in\Lambda$ with $0\le\alpha,\beta<1$, and in fact
$0<\alpha,\beta$, since a nonzero fractional multiple of a primitive ray
generator is not an integer vector; replacing $w$ by $u_0+v_0-w$ if necessary,
$\alpha+\beta\le1$, so the primitive generator $w_0$ of $w$ lies in
$\conv(0,u_0,v_0)$, has height at most $H$, and lies strictly between the two
rays. The pair $(u_0,w_0)$ has strictly smaller index: writing $w=gw_0$ with
$g\ge1$ the gcd removed from $w$, and taking determinants in lattice
coordinates of $\Lambda$, the index of $(u_0,w_0)$ is
$|\det(u_0,w_0)|=|\det(u_0,w)|/g=\beta D/g<D$. Repeating the step
on the smaller cone terminates with a basis $u,v$ of $\Lambda$ inside $K$,
both of height at most $H$; positive combinations of $u$ and $v$ are strictly
ordered, because no ordering functional vanishes on the whole kernel plane.
Now take the least integer $a\ge1$ such that $w=u+av$ has height greater
than $qH$; then $qH<\height(w)\le qH+H$. Its lattice coordinates $(1,a)$ show
that $w$ is primitive in $\Lambda$, and saturation shows that it is primitive
in $\Z^{m-1}$. Every original $q$-relation vanishes on $w$. If a new
independent $q$-relation vanished on $w$, appending it to $M$ would give a
rank-$(m-2)$ matrix whose common kernel is the line of $w$; the primitive
cofactor generator of that line has height at most $qH$ by
Lemma~\ref{M-lem:mass-minor}, contradicting the height of $w$. Thus $w$ is an
order-preserving $q$-model of diameter at most $(q+1)H$.
\end{Sremark}

\begin{Sremark}[why the lattice step stops at $d=2$]\label{S-rem:empty-tetrahedron}
The small-height unimodular-basis lemma used in
Remark~\ref{S-rem:lattice-proof} is false for arbitrary rank-three cones. For
any integer $H\ge2$ consider the cone in $\R^3$ generated by
\[
 u=(1,0,H),\qquad v=(0,1,H),\qquad w=(1,1,H),
\]
with the last coordinate as height. Every nonzero integer point of the cone
of height at most $H$ is one of $u,v,w$: at height $z$ with $0<z<H$ the cone
inequalities give $0\le x,y\le z/H<1$ and $x+y\ge z/H>0$, impossible for
integers $x,y$, and at height $H$ the only choices are $(x,y)=(1,0),(0,1),(1,1)$.
The determinant of the three primitive generators has absolute value $H$.
Thus the cone contains no lattice basis whose vectors all have height at most $H$.
The linearly independent rays used in Theorem~\ref{M-thm:flag-model}
avoid this empty-tetrahedron obstruction.
\end{Sremark}


\begin{thebibliography}{99}

\bibitem{M-ABC18} G.~Amirkhanyan, A.~Bush and E.~Croot,
\emph{Order-preserving Freiman isomorphisms}, Integers \textbf{18}
(2018), A8, 18 pp.; arXiv:1409.8535.

\bibitem{M-FB16} A.~R.~Freedman and T.~C.~Brown,
\emph{Sequences on sets of four numbers},
Integers \textbf{16} (2016), A33.

\bibitem{M-Fre73} G.~A.~Freiman,
\emph{Foundations of a Structural Theory of Set Addition},
Translations of Mathematical Monographs 37, American Mathematical Society,
Providence, 1973.

\bibitem{M-GR09} A.~Geroldinger and I.~Z.~Ruzsa,
\emph{Combinatorial Number Theory and Additive Group Theory},
Advanced Courses in Mathematics CRM Barcelona, Birkh\"auser, Basel, 2009.

\bibitem{M-Gry13} D.~J.~Grynkiewicz,
\emph{Structural Additive Theory},
Developments in Mathematics 30, Springer, 2013.

\bibitem{M-KL00} S.~V.~Konyagin and V.~F.~Lev,
\emph{Combinatorics and linear algebra of Freiman's isomorphism},
Mathematika \textbf{47} (2000), 39--51.

\bibitem{M-Nat18} M.~B.~Nathanson,
\emph{MSTD sets and Freiman isomorphisms},
Funct. Approx. Comment. Math. \textbf{58} (2018), no.~2, 187--205;
arXiv:1609.04578.

\bibitem{M-Nathanson25} M.~B.~Nathanson,
\emph{The third positive element in the greedy $B_h$-set},
Palestine Journal of Mathematics \textbf{14} (2025), 213--216;
arXiv:2310.14426v3.

\bibitem{M-Nat26} M.~B.~Nathanson,
\emph{Compression and complexity for sumset sizes in additive number theory},
J. Number Theory \textbf{281} (2026), 321--343;
DOI \href{https://doi.org/10.1016/j.jnt.2025.09.025}{10.1016/j.jnt.2025.09.025}; arXiv:2505.20998.

\bibitem{M-OBr25} K.~O'Bryant,
\emph{On Nathanson's triangular number phenomenon},
arXiv:2506.20836 (2025).

\bibitem{M-OEIS} The OEIS Foundation,
\emph{The On-Line Encyclopedia of Integer Sequences},
sequence A227589, \url{https://oeis.org/A227589}.

\bibitem{M-Shin26} H.~Shin,
\emph{Iterated-sumset spectra: The complete exponent law and its rank geometry},
\href{https://arxiv.org/abs/2609.01690v1}{arXiv:2609.01690v1},
1 September 2026.

\bibitem{M-Spec26} E.~Zhang,
\emph{The additive-square-free spectrum below 139 for four-element alphabets},
\href{https://arxiv.org/abs/2610.02215}{arXiv:2610.02215} (2026).

\bibitem{M-Shallit26} J.~Shallit, \emph{An additive-square-free word over an alphabet of size 6},
\href{https://arxiv.org/abs/2611.24441}{arXiv:2611.24441} (2026).

\bibitem{M-Sumset26} E.~Zhang,
\emph{Exact diameter compression for four-point sumset cardinalities},
\href{https://arxiv.org/abs/2610.02216}{arXiv:2610.02216} (2026).

\bibitem{M-CL07} E.~S.~Croot III and V.~F.~Lev,
\emph{Open problems in additive combinatorics},
in \emph{Additive Combinatorics}, CRM Proceedings and Lecture Notes
\textbf{43}, American Mathematical Society, 2007, 207--233.

\bibitem{M-Dim02} A.~Dimitromanolakis,
\emph{Analysis of the Golomb Ruler and the Sidon Set Problems, and
Determination of Large, Near-Optimal Golomb Rulers},
Technical University of Crete, 2002, Table~2.2.
\url{https://www.cs.toronto.edu/~apostol/golomb/main.pdf}.

\end{thebibliography}

\begin{thebibliography}{9}

\bibitem{S-OEIS} The OEIS Foundation,
\emph{The On-Line Encyclopedia of Integer Sequences},
sequence A227589, \url{https://oeis.org/A227589}.

\bibitem{S-Shin26} H.~Shin,
\emph{Iterated-sumset spectra: The complete exponent law and its rank geometry},
\href{https://arxiv.org/abs/2609.01690v1}{\texttt{arXiv:2609.01690v1}},
1 September 2026.

\end{thebibliography}
\end{document}